\documentclass[1 [leqno,11pt]{amsart}
\usepackage{amssymb, amsmath}
\usepackage{enumerate}
\def\longformule#1#2{
\displaylines{ \qquad{#1} \hfill\cr \hfill {#2} \qquad\cr } }
\def\inte#1{
\displaystyle\mathop{#1\kern0pt}^\circ }

\let\pa=\partial
\let\al=\alpha

\let\d=\partial
\let\e=\varepsilon

\let\r=\rho
\let\s=\sigma
\let\f=\frac

\let\p=\psi
\let\om=\omega

\let\D=\Delta

\let\Om=\Omega
\let\wt=\widetilde

\def\cB{{\mathcal B}}
\def\cC{{\mathcal C}}

\def\cF{{\mathcal F}}

\def\cH{{\mathcal H}}

\def\cM{{\mathcal M}}

\def\cS{{\mathcal S}}

\def\cW{{\mathcal W}}

\def\grad{\nabla}
\def\la{\lambda}

\def\ov{\overline}

\renewcommand{\div}{{\rm div}\,}
\newcommand{\loc}{{\rm loc}\,}

\newcommand{\Id}{{\rm Id}\,}
\newcommand{\Supp}{{\rm Supp}\,}

\newcommand\ds{\displaystyle}

\newcommand{\Prod}{\displaystyle \prod}

\def\dH{\dot{H}}
\def\dB{\dot{B}}

\def\h{\frak h}

\def\virgp{\raise 2pt\hbox{,}}
\def\cdotpv{\raise 2pt\hbox{;}}

\def\eqdefa{\buildrel\hbox{\footnotesize def}\over =}
\def\Id{\mathop{\rm Id}\nolimits}

\def\C{\mathop{\mathbb C\kern 0pt}\nolimits}
\def\DD{\mathop{\mathbb D\kern 0pt}\nolimits}
\def\EE{\mathop{{\mathbb E \kern 0pt}}\nolimits}
\def\K{\mathop{\mathbb K\kern 0pt}\nolimits}
\def\N{\mathop{\mathbb N\kern 0pt}\nolimits}
\def\Q{\mathop{\mathbb Q\kern 0pt}\nolimits}
\def\R{\mathop{\mathbb R\kern 0pt}\nolimits}
\def\SS{\mathop{\mathbb S\kern 0pt}\nolimits}
\def\ZZ{\mathop{\mathbb Z\kern 0pt}\nolimits}
\def\TT{\mathop{\mathbb T\kern 0pt}\nolimits}
\def\PP{\mathop{\mathbb P\kern 0pt}\nolimits}

\newcommand{\Z}{{\ZZ}}

\def\dive{\mathop{\rm div}\nolimits}

\def\Supp{\mathop{\rm Supp}\nolimits\ }

\def\no{\noindent}
\def\na{\nabla}
\def\p{\partial}

\def\h{{\rm h}}
\def\v{{\rm v}}

\def\th{\theta}

\newcommand{\beq}{\begin{equation}}
\newcommand{\eeq}{\end{equation}}
\newcommand{\ben}{\begin{eqnarray}}
\newcommand{\een}{\end{eqnarray}}
\newcommand{\beno}{\begin{eqnarray*}}
\newcommand{\eeno}{\end{eqnarray*}}
\newcommand{\andf}{\quad\hbox{and}\quad}
\newcommand{\with}{\quad\hbox{with}\quad}
\newtheorem{defi}{Definition}[section]
\newtheorem{thm}{Theorem}[section]
\newtheorem{lem}{Lemma}[section]
\newtheorem{rmk}{Remark}[section]

\newtheorem{prop}{Proposition}[section]

\begin{document}
\title[Global regularity of 2-D density patch]
{On the global  regularity of 2-D density patch for inhomogeneous
incompressible viscous flow}
 \author[X. Liao]{Xian Liao}
\address [X. Liao]%
{Academy of Mathematics $\&$ Systems Science and  Hua Loo-Keng Key
Laboratory of Mathematics, The Chinese Academy of Sciences, CHINA}
\email{xian.liao@amss.ac.cn}
\author[P. ZHANG]{Ping Zhang}%
\address[P. Zhang]
 {Academy of
Mathematics $\&$ Systems Science and  Hua Loo-Keng Key Laboratory of
Mathematics, The Chinese Academy of Sciences, CHINA}
\email{zp@amss.ac.cn}

\date{3/17/2015}
\maketitle

\begin{abstract} Toward P.-L. Lions' open question  in \cite{Lions96} concerning the propagation of  regularity for
density patch, we establish  the global existence of solutions to
the 2-D inhomogeneous incompressible Navier-Stokes system with
initial density given by $(1-\eta){\bf 1}_{\Om_0}+{\bf 1}_{\Om_0^c}$
for some small enough constant $\eta$ and  some $W^{k+2,p}$ domain
$\Om_0,$  and with initial vorticity belonging to $L^1\cap L^p$ and
with appropriate tangential regularities. Furthermore, we  prove
that the regularity of
 the domain $\Om_0$ is preserved by time evolution.
\end{abstract}

\par $\qquad\quad\ \ $
striated distributions,  Littlewood-Paley Theory


 \setcounter{equation}{0}
\section{Introduction}

We  consider the following inhomogeneous incompressible
Navier-Stokes equations in two space dimension:
\begin{equation}\label{inNS}
 \left\{\begin{array}{l}
\displaystyle \d_t\rho+\div(\rho v)=0,\qquad (t,x)\in\R^+\times\R^2,\\
\displaystyle\d_t(\rho v)+\dive(\r v\otimes v)-\Delta v+\nabla\pi=0,\\
\displaystyle \div v=0,\\
\displaystyle  (\r, v)|_{t=0}=(\r_0, v_0)
\end{array}\right.
\end{equation}
where $\rho\in\R^+,$ $v\in\R^2$ and $\pi\in\R$ stand for the
density, velocity field  and pressure of the fluid, respectively.
 The
above system describes a  fluid that is incompressible but has
nonconstant density. Basic examples are mixture of incompressible
and non reactant flows, flows with complex structure (e.g. blood
flow or model of rivers), fluids containing a melted substance, etc.
 \smallbreak
 A lot of recent works have been dedicated to the mathematical study
of the above system. Global weak solutions with finite energy have
been constructed by J. Simon in \cite{Simon} (see also the book by
P.-L. Lions \cite{Lions96} for the variable viscosity case). In the
case of smooth data with no vacuum, the existence of strong unique
solutions goes back to the work of O. Ladyzhenskaya and V.
Solonnikov in \cite{LS}. More recently, R. Danchin \cite{D1}
established the well-posedness of the above system in the whole
space $\R^d$ in the so-called \emph{critical functional framework}
for small perturbations of some positive constant density. The basic
idea is to use functional spaces (or norms) that have the same
\emph{scaling invariance} as \eqref{inNS}, namely
\begin{equation*}
(\rho,v,\pi)(t,x)\longmapsto (\rho,\lambda v,\lambda^2\pi)
(\lambda^2 t,\lambda x),\qquad (\rho_0,v_0)(x)\longmapsto
(\rho_0,\lambda v_0)(\lambda x).
\end{equation*}
More precisely, in \cite{D1}, global well-posedness was established
assuming that for some small enough constant $c,$ one has
$$\|\rho_0-1\|_ {\dot B^{\frac d2}_{2,\infty}(\R^d)\cap L^\infty(\R^d)}+\|v_0\|_{\dot B^{-1+\frac d2}_{2,1}(\R^d)}\leq c.$$
Above $\dot B^{\sigma}_{p,r}(\R^d)$ stands for a homogeneous Besov
space on $\R^d.$
 This result
was extended to more general Besov spaces  by H. Abidi in
\cite{Abidi}, and by H. Abidi and M. Paicu in \cite{AP}, and to more
general smallness condition by the second author and M. Paicu in
\cite{PZ2}. The smallness assumption on the initial density was
removed recently in \cite{AGZ2,AGZ3}.

\smallbreak

 Given that in all those works the density has to be at
least in the Besov space $\dot B^{\f dp}_{p,\infty}(\R^d),$
  one cannot capture  discontinuities  across an hypersurface. In effect,
  the Besov regularity of the characteristic function of a smooth domain is only $\dot B^{\f1p}_{p,\infty}(\R^d).$
  Therefore, those results do not apply to a  mixture of two fluids with different densities.

\smallbreak In \cite{DM12},  R. Danchin and P. Mucha proved the
global existence and uniqueness  of solutions to \eqref{inNS} for
any data $(\rho_0,u_0)$ such that for some $p\in[1,2d[$ and small
enough constant $c,$ there holds
\begin{equation}
\label{eq:small} \|\rho_0-1\|_{\cM(\dot B^{-1+\frac
dp}_{p,1}(\R^d))}+\|v_0\|_{\dot B^{-1+\frac dp}_{p,1}(\R^d)}\leq c.
\end{equation}
Above, $\|\cdot\|_{\cM(\dot B^{-1+\frac dp}_{p,1}(\R^d))}$ denotes
the \emph{multiplier norm} associated to the Besov space $\dot
B^{-1+\frac dp}_{p,1}(\R^d),$ which turns out to be finite for
characteristic functions of $C^1$ domains whenever $p>d-1.$
Therefore, initial densities with a discontinuity across an
interface may be considered (although the jump has to be small owing
to \eqref{eq:small}). In fact, even  $\rho_0$ is only bounded and
bounded away from zero,  J. Huang, M. Paicu and the second author in
\cite{HPZ} could solve the global existence of solutions to
\eqref{inNS}  in a critical functional framework, and uniqueness was
obtained if assuming slightly more regularity for the velocity
field. A half space setting of this problem was solved by R. Danchin
and the second author  lately in \cite{DZ14}.

 \smallbreak   A natural question to ask is whether it is still possible to
 propagate higher order regularities of the interface of the fluids.
 Especially, P.-L. Lions proposed the following open question in
 \cite{Lions96}: suppose the initial density $\r_0=1_{D}$ for some
 smooth domain $D,$ Theorem 2.1 of \cite{Lions96} yields at least
 one global weak solution $(\r,v)$ of \eqref{inNS} such that for all
 $t\geq 0,$ $\r(t)=1_{D(t)}$ for some set $D(t)$ with ${\rm vol}(D(t))={\rm vol}(D).$
 Then whether or not the regularity of $D$ is preserved by the time
 evolution. Since this problem is very sophisticated due to the
 appearance of vacuum, as a first step toward this question, here we aim at
 solving the global well-posedness of \eqref{inNS} with  the initial
density $\r_0(x)=(1-\eta){\bf 1}_{\Om_0}+{\bf 1}_{\Om_0^c} $ for
some simply connected $W^{k+2,p}(\R^2)$ domain $\Om_0$  and  $\eta$
being sufficiently small.

More precisely, let  $\Omega_0$ be a simply connected $W^{k+2,p},$
for $k\geq 1,$ $p\in ]2,4[,$ bounded domain. Let $f_0\in
W^{k+2,p}(\R^2)$ such that $ \d\Omega_0=f_0^{-1}(\{0\})$ and $\nabla
f_0 $ does not vanish on $\d\Omega_0. $ Then we can parametrize
$\d\Omega_0$ as \beq\label{S3eq19} \gamma_0: \cS^1\mapsto \d\Omega_0
\hbox{ via } s\mapsto \gamma_0(s) \hbox{ with
}\d_s\gamma_0(s)=\nabla^\bot f_0(\gamma_0(s)). \eeq We take the
initial density $\r_0$ and the initial vorticity $\omega_0$ of
\eqref{inNS} as follows \beq\label{S3eq18}
\begin{split}
 &  {\rho_0}={(1-\eta)\bf 1}_{\Om_0}+{\bf 1}_{\Om_0^c}, \andf \om_0\in L^1(\R^2) \with \p_{X_0}^\ell\om_0\in
 W^{-\f{\ell}k\e,p}(\R^2)
    \end{split}\eeq
for some $\e\in ]0,1[, \ell=0,\cdots, k,$ and  $X_0=\na^\perp f_0,$
$\p_{X_0}\om_0\eqdefa X_0\cdot\na\om_0.$ \smallbreak

The main motivation for us to solve the above problem comes from
\cite{Chemin88, Chemin91, Chemin93, Chemin98}. Let us mention that
by using the idea of conormal distributions or striated
distributions, J.-Y. Chemin \cite{Chemin91, Chemin93} (see also
\cite{BC93}) proved the global regularities for the vortex patch
problem of 2-D incompressible Euler system. One may also check the
following papers as extensions of the results established in
\cite{Chemin91, Chemin93}: P. Gamblin and X. Saint Raymond
\cite{GS95} and the second author and Q. Qiu \cite{ZQ} for the 3-D
extensions, and  R. Danchin \cite{Danchin97} and T. Hmidi
\cite{Hmidi05, Hmidi06} for the viscous extensions.
 We mention
that the tools in \cite{Chemin88, Chemin91, Chemin93, Chemin98} will
 play an essential role in this paper.

\smallbreak The main result of this paper states as follows:

\begin{thm}\label{thmmain}
{\sl Let the initial data $(\r_0, v_0)$ be given by \eqref{S3eq18}
for $\eta$ sufficiently small. Then \eqref{inNS} has a unique global
solution $(\r, v)$ so that $\r(t,x)={(1-\eta)\bf 1}_{\Om(t)}+{\bf
1}_{\Om(t)^c}$ for some    simply connected $W^{k+2,p}$ domain
$\Om(t).$ }
\end{thm}

\begin{rmk} For $\Om_0$ given by \eqref{S3eq19}, it is easy to
observe that $\p_{X_0}^\ell {\bf 1}_{\Om_0}(x)=0$ for any $\ell\in
\{1,\cdots, k\}.$ Hence in particular, Theorem \ref{thmmain} ensures
the global well-posedness of \eqref{inNS} with initial density
$\r_0=(1-\eta){\bf 1}_{\Om_0}+{\bf 1}_{\Om_0^c}$ and initial
vorticity $\om_0={\bf 1}_{\Om_0}(x).$ Furthermore, the geometric
structure of $\Om_0$ will be persisted for all time. \end{rmk}

\begin{rmk} \begin{itemize}

\item[(1)] Compared with the vector field method in \cite{Chemin88, Chemin91, Chemin93,
Chemin98}, here new difficulties arise due to the non-commutative
property of the vector field with the viscosity term and the
estimate of the pressure function.

\item[(2)] Furthermore due to  the transport equation of the
density function in \eqref{inNS}, the estimate of the pressure
function is more subtle than that in \cite{Danchin97,Hmidi05,
Hmidi06}.

\item[(3)] Technically since we shall use the maximal regularity estimate
for evolutionary  Stokes operator, we shall use striated
distributions in the framework of the Sobolev space, $W^{k,p},$
which is of Triebel-Lizorkin type space, yet the striated
distributions in the previous papers are in the framework of Besov
type spaces.

\end{itemize}
\end{rmk}

\setcounter{equation}{0}
\section{Basic facts on Littlewood-Paley analysis}\label{appa}

For the convenience of the reader, we recall some basic facts on
Littlewood-Paley theory, and one may check Chapter 2 of \cite{BCD}
for more details.
 Let us briefly explain how it may be built in the
case $x\in\R^d$ (see e.g. \cite{BCD}). Let $\chi,\varphi$ be smooth
functions  supported in the ball $\mathcal{B}\eqdefa \{ \xi\in\R^d,
|\xi|\leq\frac{4}{3}\}$ and the ring $\mathcal{C}\eqdefa \{
\xi\in\R^d,\frac{3}{4}\leq|\xi|\leq\frac{8}{3}\}$ respectively, such
that
\begin{equation*}
 \sum_{j\in\Z}\varphi(2^{-j}\xi)=1 \quad\hbox{for}\quad \xi\neq 0\andf \chi(\xi)+\sum_{j\geq 0}
 \varphi(2^{-j}\xi)=1 \quad\hbox{for}\quad \xi\in\R^d.
\end{equation*}
Then for $u\in{\mathcal S}'(\R^d),$ we set
\begin{equation}\label{dydic}
\begin{split}
&\forall j\in\Z,\quad \dot\Delta_j u=\varphi(2^{-j}D)
u\hspace{1cm}\mbox{and} \hspace{1cm} \dot S_ju=\chi(2^{-j}D)u,\\
&\forall j\in\N,\quad \Delta_j u=\varphi(2^{-j}D) u,\quad
\Delta_{-1} u=\chi(D) u\quad  \mbox{and} \quad S_ju=\sum_{j'\leq
j-1}\D_{j'}u.
\end{split}
\end{equation}
 The dyadic blocks satisfies
the property of almost orthogonality:
\begin{equation}\label{Pres_orth}
\dot\Delta_k\dot\Delta_j u\equiv 0 \quad\mbox{if}\quad| k-j|\geq 2
\quad\mbox{and}\quad \dot\Delta_k( \dot S_{j-1}u\dot \Delta_j u)
\equiv 0\quad\mbox{if}\quad| k-j|\geq 5.
\end{equation}

We first recall  the definition of homogeneous Besov spaces.
\begin{defi}\label{def1.1}
{\sl  Let $(p,r)\in[1,+\infty]^2,$ $s\in\R$ and $u\in{\mathcal
S}_h'(\R^d),$ which means that $u$ is in~$\cS'(\R^d)$ and
satisfies~$\ds\lim_{j\to-\infty}\|\dot S_ju\|_{L^\infty}=0$. We set
$$
\|u\|_{\dot B^s_{p,r}}\eqdefa \Big\| \Bigl(2^{js}\|\dot\Delta_j
u\|_{L^{p}}\Bigr)\Big\|_{\ell ^{r}}.
$$
\begin{itemize}

\item
For $s<\frac{d}{p}$ (or $s=\frac{d}{p}$ if $r=1$), we define $\dot
B^s_{p,r}(\R^d)\eqdefa \big\{u\in{\mathcal S}'(\R^d)\;\big|\; \|
u\|_{\dot B^s_{p,r}}<\infty\big\}.$

\item
If $k\in\N$ and $\frac{d}{p}+k\leq s<\frac{d}{p}+k+1$ (or
$s=\frac{d}{p}+k+1$ if $r=1$), then $\dot B^s_{p,r}(\R^d)$ is
defined as the subset of distributions $u\in{\mathcal S}'(\R^d)$
such that $\partial^\beta u\in\dot B^{s-k}_{p,r}(\R^d)$ whenever
$|\beta|=k.$
\end{itemize}}
\end{defi}

We also need the following version of Bernstein Lemma:

\begin{lem}\label{lem2.1} {\sl Let $\cB$ be a ball   and $\cC$ a ring of $\R^d.$
 A constant $C$ exists so
that for any positive real number $\lambda$, any non negative
integer $k$, any smooth homogeneous function $\sigma$ of degree $m,$
and any couple of real numbers $(a, \; b)$ with $ b \geq a \geq 1,$
there hold
\begin{equation}
\begin{split}
&\Supp \hat{u} \subset \lambda \mathcal{B} \Rightarrow
\sup_{|\alpha|=k} \|\pa^{\alpha} u\|_{L^{b}} \leq  C^{k+1}
\lambda^{k+ d(\frac{1}{a}-\frac{1}{b} )}\|u\|_{L^{a}},\\
& \Supp \hat{u} \subset \lambda \mathcal{C} \Rightarrow
C^{-1-k}\lambda^{ k}\|u\|_{L^{a}}\leq
\sup_{|\alpha|=k}\|\partial^{\alpha} u\|_{L^{a}} \leq
C^{1+k}\lambda^{ k}\|u\|_{L^{a}},\\
& \Supp \hat{u} \subset \lambda \mathcal{C} \Rightarrow \|\sigma(D)
u\|_{L^{b}}\leq C_{\sigma, m} \lambda^{ m+d(\frac{1}{a}-\frac{1}{b}
)}\|u\|_{L^{a}}\\
& \Supp \hat{u} \subset \lambda \mathcal{C} \Rightarrow
 \|\textrm{e}^{t\Delta}
u\|_{L^{a}}\leq C \textrm{e}^{-c\lambda^2 t}\|u\|_{L^{a}}.
\end{split}\label{2.1}
\end{equation}}
\end{lem}

Finally let us recall  Bony's decomposition in the inhomogeneous
context from \cite{Bo81}:
\begin{equation}
\begin{split}
&uv=T_u v+T_v u+R(u,v), \end{split}\label{bony}
\end{equation}
where
\begin{equation*}
\begin{split}
&{T}_u v\eqdefa\sum_{j \in \mathbb{Z}}S_{j-1}u\Delta_j v,\quad
{R}(u,v)\eqdefa\sum_{j\in\Z}\Delta_j u \widetilde{
  \Delta}_{j}v\with
\widetilde{\Delta}_{j}v\eqdefa \sum_{|j'-j|\leq 1}\Delta_{j'}v.
\end{split}
\end{equation*}
We sometime will also use Bony's decomposition in homogeneous
context.

 \setcounter{equation}{0}
\section{ Ideas of the proof and structure of the paper}

 Because we shall consider only
perturbations of the reference density $1,$ it is natural to set
$a\eqdefa 1/\rho-1$ so that System \eqref{inNS} translates into
\begin{equation}\label{INS}
 \quad\left\{\begin{array}{lll}
\displaystyle \pa_t a + v \cdot \grad a=0 \qquad (t,x)\in\R^+\times\R^2,  \\
\displaystyle \pa_t v + v \cdot \grad v+ (1+a)(\grad\pi-\D v)=0, \\
\displaystyle \dive\, v = 0, \\
 \displaystyle (a, v)|_{t=0}=(a_0, v_{0}).
\end{array}\right.
\end{equation}

Before proceeding, let us  recall  the striated distribution spaces
with respect to the vector field $X$ from  \cite{Chemin88,
Chemin91}:
\begin{equation}\label{S1eq7}
C^{\sigma}_{\delta}(X,k) \eqdefa\Bigl\{ f\in
C^{\sigma}(\R^2)\,\bigl|\,
 \|f\|^{\sigma,k}_{\delta,X}
\eqdefa\sum_{\ell=0}^k
 \|(T_{X\cdot\nabla} )^\ell f\|_{C^{\sigma-\ell\delta }}
 <+\infty \Bigr\},
 \quad \sigma\in \R^\ast\backslash\{1\},
 \end{equation} where $T_ab$ denotes the para-product of $b$ by
 $a$  (see \eqref{bony}).
 And  we call a vector field $X$ is $\delta,k$-regular if
\beq \label{S6eq4}
 \|X\| ^{1-\delta,k-1}_{\delta,X}<\infty, \with
  \delta\in ]0,1[
 \andf
 k\delta <1.
\eeq However  as in \cite{ DZ14, HPZ}, the functional framework in
this work for solving \eqref{INS} is motivated by classical maximal
regularity estimates for the evolutionary Stokes system, so that it
is more natural to introduce the striated distribution space in the
framework of Sobolev space $W^{s,p}$
\begin{equation} \label{S6eq5}
W^{s,p}_{ \delta}(X,k)\eqdefa \Bigl\{ f\in W^{s,p}(\R^2)\,\bigr|\,
  |f |^{s,p,k}_{\delta,X}
\eqdefa\sum_{\ell=0}^k
 \|(T_{X\cdot\nabla} )^\ell f\|_{W^{s-\ell\delta,p }}
 <+\infty \Bigr\} .
 \end{equation}
Corresponding to \eqref{S6eq5}, we also define
\begin{equation} \label{S1eq3}
\cW^{s,p}_{ \delta}(X,k) \eqdefa\Bigl\{ f\in W^{s,p}(\R^2)\,\Bigr|\,
  |f\|^{s,p,k}_{\delta,X}
\eqdefa \sum_{\ell=0}^k
 \|\p_X^\ell f\|_{W^{s-\ell\delta,p }}
 <+\infty \Bigr\}.
 \end{equation} Here and in all that follows, we always denote
 $\p_Xf={X\cdot\nabla}f.$
 We shall prove in Lemma \ref{lem:prop} below that the function  spaces $W^{s,p}_{
 \delta}(X,k)$ and $\cW^{s,p}_{ \delta}(X,k)$ are the same provided
 that the vector field $X$ is $\delta,k$-regular, $s\in ]k\delta,1[$ or  $s\in ]k\delta-1,1[$ if $\dive X=0.$

 \smallbreak

 We shall first prove the following more general result:

 \begin{thm}\label{thm1}
 {\sl Let $p\in ]2,4[,$ let $X_0=(X_0^1,X_0^2)$ and $(a_0,v_0)$
 satisfy $\dive X_0=0,$
  $\p_{X_0}^{\ell-1}X_0\in W^{2,p},$
  $\p_{X_0}^\ell a_0 \in L^\infty$ for $\ell\in \{1,\cdots, k\}$ and
  $a_0\in L^2, v_0\in \cW^{1,p}_{
 \varepsilon/k}(X_0,k)$ for some $\varepsilon\in ]0,1[$.
  Then there exists a positive constant $c_0$ so that
 whenever $\|a_0\|_{L^\infty}\leq c_0,$ there exists a divergence free vector field $X(t)$
 which satisfies \beq\label{S1eq2}
 \left\{\begin{array}{l}
\displaystyle \p_tX+v\cdot X=X\cdot\na v,\\
\displaystyle X(0,x)=X_0(x),
\end{array}\right. \eeq
 $\p_X^{\ell-1}X\in L^\infty_\loc(\R^+; W^{2,p})$ for $\ell\in \{1,\cdots, k\},$  and \eqref{INS} has a
 unique global
 solution $(a, v)$ with
 \beq\label{S1eq4}
 \begin{split}
&\|(\d_t \d_X^{\ell } v,\Delta\d_X^{\ell} v, \nabla\d_X^{\ell
}\pi)\|_{L^{r_{\ell }}_t(L^p)} +\|\d_X^{\ell
}a\|_{L^\infty_t(L^\infty)}+\|\nabla\d_X^{\ell } v\|_{L^{r_{\ell
}}_t(L^\infty) \cap
L^{s_{\ell }}_t(L^p)}\\
&\qquad\qquad+\|\d_X^{\ell } v\|_{L^{s_{\ell }}_t(L^\infty)} +
\|\d_t \d_X^{\ell-1}X\|_{L^{s_{\ell }}_t(W^{1,p})} + \|\d_X^{\ell-1}
X\|_{L^\infty_t(W^{2,p})}\leq\cH_\ell(t),
\end{split}
\eeq for  $\ell\in \{0,\cdots, k\}$ (with the convention:
$\d_X^{-1}X\eqdefa 0$). Here and in all that follows, we always
denote  $r_\ell, s_\ell$ to be any numbers with $r_\ell \in
\bigl]1,\f{2k}{k+\ell\e}\bigr[$ and $s_\ell\in
\bigl]2,\f{2k}{\ell\e}\bigr[,$   and $ \cH_\ell(t), \ell=0,\cdots,
k,$ to be positive increasing functions of $t$ which depend on the
initial data, $r_\ell, s_\ell,$  and which may change from line to
line.}\end{thm}

As an application of the above theorem, we consider the propagation
of the boundary regularity for the 2-D density patch's problem of
\eqref{inNS}.

\begin{thm}\label{thm2}
{\sl Let the initial data $(\r_0, v_0)$ be given by \eqref{S3eq18}
for $\eta$ sufficiently small. Then \eqref{inNS} has a unique global
solution $(\r, v)$ so that $\r(t,x)={(1-\eta)\bf 1}_{\Om(t)}+{\bf
1}_{\Om(t)^c}$ for some    simply connected $W^{k+2,p}$ domain
$\Om(t),$ which is determined by the level surface of $f\in
L^\infty_\loc(\R^+;W^{k+2,p})$ so that with
  $X\eqdefa \na^\perp f,$ there holds \eqref{S1eq4}.}
\end{thm}

It is easy to observe that Theorem \ref{thmmain} is a direct
consequence of Theorem \ref{thm2}.\smallskip

As an illustration to
 the main idea of the proof to Theorem \ref{thm2}, we first
 sketch the proof of Theorem \ref{thm2} for the case
$k=1$  by assuming Theorem \ref{thm1}. \smallskip

\no$\bullet$ \underline{Sketch of the proof of Theorem \ref{thm2}
for $k=1.$} Let $\psi(t,\cdot)$ be the flow associated with the
velocity vector field $v,$ that is
\begin{equation}\label{S1eq6}
 \left\{\begin{array}{l}
\displaystyle \f{d}{dt}\psi(t,x)=v(t,\psi(t,x)),\\
\displaystyle \psi(0,x)=x.
\end{array}\right.
\end{equation}
Then by virtue of Theorem \ref{thm1}, one has $$
\psi(t,\cdot)-\Id\in L^\infty_\loc (\R^+; W^{2,p}),
$$ and hence $\Om(t)\eqdefa \psi(t,\Omega_0)$ is  a
 $W^{2,p}$ domain. Moreover, it follows from the transport equation of \eqref{inNS} that
 \beq \label{density} \r(t,x)={(1-\eta)\bf 1}_{\Om(t)}+{\bf
1}_{\Om(t)^c}. \eeq
   Next we are going to prove that
$\Omega(t)$ is
 of class $W^{3,p}.$ As a matter of fact,
 let $\d\Omega(t)$
 be the level surface of $f(t,\cdot).$ Then $f$ solves
\beno
 \left\{\begin{array}{l}
\displaystyle \p_tf+v\cdot\na f=0,\\
\displaystyle f(0,x)=f_0(x),
\end{array}\right. \eeno
which implies that $X(t,x)\eqdefa \na^\perp f(t,x)$ solves
\eqref{S1eq2}. Thus it follows from \eqref{S1eq2} and \eqref{S1eq6}
that \beq \label{S3eq21}
\begin{split} &X(t,\psi(t,x))=X_0(x)\f{\p\psi(t,x)}{\p x}.
\end{split} \eeq
Moreover, due to $\dive X_0=0=\dive v,$ one has \beq\label{S1eq8}
 \p_t\dive X+v\cdot \dive X=0 \with \dive X_0=0,
 \eeq
 Hence $\dive X(t)=0,$ and
according to Theorem \ref{thm1}, we have
$$
 X\in
L^\infty_\loc(\R^+; W^{2,p}).
$$
It is easy to observe that $X\in L^\infty_\loc(L^p)$, and \beno
\begin{split} \d_{j}\p_j\bigl(X^k(t,\psi(t,x))\bigr)
=&\d_{i}\p_j\psi(t,x)\cdot\nabla X^k(t,\psi(t,x))\\
&+(\d_i\psi\otimes\d_j\psi)(t,x):\nabla^2 X^k(t,\psi(t,x)) \in
L^\infty_\loc(\R^+;L^p).
\end{split} \eeno
Therefore we deduce from \eqref{S3eq21} that \beq\label{S1eq5}
 \d_s(\psi(t,\gamma_0(s)))=X_0(\gamma_0(s))
 \f{\p\psi(t,\gamma_0(s))}{\p  x}
= X(t,\psi(t,\gamma_0(s)))\in L^\infty_\loc(\R^+; W^{2,p}),\eeq that
is, $\Omega(t)$ is  of class $W^{3,p}$.\\

Let us now present the proof of Theorem \ref{thm2} by assuming
Theorem \ref{thm1}.
\smallskip

\begin{proof}[Proof of Theorem \ref{thm2}] In view of \eqref{S3eq18},  we observe that $X_0\in W^{k+1,p}$, $k\geq 1$, $\div X_0=0$, and
moreover
\begin{equation}\label{initial:X0}
\d_{X_0}^\ell a_0\equiv 0, \quad \d_{X_0}^{\ell-1}X_0\in
W^{k+2-\ell,p}\hookrightarrow W^{2,p}, \quad \ell=1,\cdots,k.
\end{equation}
Since $p>2$ and $\na v_0=\na\na^\perp\D^{-1}\om_0,$ for any integer
$N,$ it follows from Lemma \ref{lem2.1} that  \beno
\begin{split}
\|v_0\|_{L^p}\lesssim & \sum_{ j\leq
N}2^{j\bigl(1-\f2p\bigr)}\|\dot\D_j\na
v_0\|_{L^1}+\sum_{j>N}2^{-j}\|\na\dot\D_j v_0\|_{L^p}\\
\lesssim &
2^{N\bigl(1-\f2p\bigr)}\|\om_0\|_{L^1}+2^{-N}\|\om_0\|_{L^p}.
\end{split} \eeno
Choosing the best integer $N$ in the above inequality so that \beno
2^N \sim
\biggl(\frac{\|\om_0\|_{L^p}}{\|\om_0\|_{L^1}}\biggr)^{\f{p}{2(p-1)}}
\eeno leads to \beno \|v_0\|_{L^p}\leq
C\|\om_0\|_{L^1}^{\f{p}{2(p-1)}}\|\om_0\|_{L^p}^{\f{p-2}{2(p-1)}}.
\eeno As $\nabla v_0=\nabla\nabla^\bot \Delta^{-1}\omega_0\in L^p$,
we achieve \beq\label{S6eq0} \|v_0\|_{ W^{1,p}}\leq
C\|\om_0\|_{L^1\cap L^p}.\eeq Moreover, we claim  the following
striated regularity for $v_0$
\begin{equation}\label{initial:v0}
|v_0\|_{\e/k,X_0}^{1,p,k}\leq C|\om_0|_{\e/k,X_0}^{0,p,k}
\quad\mbox{for any}\quad \e\in ]0,1[,
\end{equation} which will be proved in Section \ref{Sect7}.

Thanks to  \eqref{initial:X0} and   \eqref{initial:v0}, we deduce
from Theorem \ref{thm1} that  with initial data $(a_0,v_0)$ given by
\eqref{S3eq18}, \eqref{INS} has a unique solution $(a,v)$ satisfying
\eqref{S1eq4}. Furthermore, due to $X_0=\nabla^\bot f_0 $, it
follows from a similar derivation of \eqref{S1eq5} that
$$
\d_s^{k}(\psi(t,\gamma_0(s))) =\d_s^{k-1}\bigl(
(\d_{X_0}\psi)(t,\gamma_0(s)) \bigr) =\cdots
=(\d_{X_0}^{k}\psi)(t,\gamma_0(s)), \quad\forall k\geq 1 .
$$
Hence in order to show that $\p\Om(t)=\psi(t,\gamma_0(\p\Om_0)))\in
W^{k+2,p}$, it suffices to prove that $\d_{X_0}^k\psi\in
L^\infty_\loc(W^{2,p}),$ which is equivalent to show that
$\d_X^{k-1}X\in L^\infty_\loc(W^{2,p})$. Indeed, it follows from
 \eqref{S1eq6} and \eqref{S1eq2} that
$$
X(t,x)=(\d_{X_0}\psi)(t,\psi^{-1}(t,x)),
$$
and hence
\begin{align*}
(\d_X X)(t,x) &=(\d_{X_0}\psi)(t,\psi^{-1}(t,x)) \cdot
\nabla\psi^{-1}(t,x) \cdot (\nabla\d_{X_0}\psi)(t,\psi^{-1}(t,x))
\\
&=X_0(t,\psi^{-1}(t,x)) \cdot(\nabla\psi)(t,\psi^{-1}(t,x))
\cdot\nabla\psi^{-1}(t,x)
\cdot(\nabla\d_{X_0}\psi)(t,\psi^{-1}(t,x))
\\
&=X_0(t,\psi^{-1}(t,x)) \cdot(\nabla\d_{X_0}\psi)(t,\psi^{-1}(t,x))
\\
&=( \d_{X_0}^2\psi)(t,\psi^{-1}(t,x)).
\end{align*}
The same arguments yield   that
\begin{align}\label{claim:X}
(\d_X^{k-1} X)(t,x)=(\d_{X_0}^{k}\psi)(t,\psi^{-1}(t,x)), \quad
\forall k\geq 1,
\end{align}
which implies $ \d_{X_0}^k\psi\in L^\infty_\loc(W^{2,p}) $ if
$\d_X^{k-1} X\in L^\infty_\loc(W^{2,p}).$ And hence by virtue of
\eqref{S1eq4}, we conclude that  $\p\Om(t)\in W^{k+2,p}.$ This
completes the proof of Theorem \ref{thm2}.
\end{proof}

\no The structure of the paper is organized as follows: \smallskip

\no Section \ref{S4}, Section \ref{S5} and Section \ref{S6} are
devoted to the proof of Theorem \ref{thm1}: Section \ref{S4} is
devoted to the proof of the Estimate \eqref{S1eq4} for $\ell=0$;
in Section \ref{S5} we prove the Estimate \eqref{S1eq4} for the case $\ell=1$; and in Section \ref{S6}
we first derive the Estimate \eqref{S1eq4} for  the general case $\ell=2,\cdots,k$, and then Theorem \ref{thm1} will be proved at the end.\\
We concentrate on the proof of \eqref{initial:v0} in Section
\ref{Sect7} and the appendix will be devoted to the proofs of Lemma
\ref{lem:Rq} and Lemma \ref{lem:prop}. \smallskip

Let us complete this section with the notations we are going to use
in this context.

\smallbreak \noindent{\bf Notations:} $C^\s$ denotes the classical
H\"older space, $W^{s,p}$ stands for the
 Sobolev space with norm given by
$\|f\|_{W^{s,p}}\eqdefa \|(Id-\D)^{\f{s}2}f\|_{L^p}.$ We denote
$[A;B]=AB-BA$ the commutator of $A$ and $B.$ For $X$ a Banach space
and $I$ an interval of $\R,$ we denote by
 $L^q(I;\,X)$
 the set of measurable functions on $I$ with values in
$X,$ such that $t\longmapsto\|f(t)\|_{X}$ belongs to $L^q(I).$ For
$a\lesssim b$, we mean that there is a uniform constant $C,$ which
may be different on different lines, such that $a\leq Cb$.

 \setcounter{equation}{0}
\section{The preliminary estimates}\label{S4}

In this section, we shall prove the Estimate \eqref{S1eq4} for
$\ell=0,$  which will be based on the energy estimate of the System
\eqref{inNS}.

\begin{lem}\label{S2lem1}
{\sl Let $p\in ]2,4[,$ $r\in ]1,2[, s\in ]2, \infty]$ and $q\in
\bigl]\f{2p}{p+1}, 2p\bigr[.$ Let $v_0\in W^{1,p}$ and $
v_L(t)\eqdefa e^{t\D}v_0.$ Then there holds \beq\label{S2eq0} \|\D
v_L\|_{L^{r}_t(L^p)}+\bigl\|\na v_L\bigr\|_{L^{s}_t(L^p)}+\|\na
v_L\|_{L^q_t(L^{2p})}\leq C\|v_0\|_{W^{1,p}}. \eeq}
\end{lem}

\begin{proof} Since $p\in ]2,4[,$ by virtue of Theorem 2.41 of \cite{BCD}, we have
$v_0\in W^{1,p}\hookrightarrow \dot{B}^0_{p,p}\cap \dB^1_{p,p},$
where $\dot B^\s_{p,q}$ denotes the homogeneous Besov space (see
Definition \ref{def1.1}). Hence according to Theorem 2.34 of
\cite{BCD} (the definition of Besov spaces with negative indices
through heat kernel) that \beno &&\bigl\|t^{\f12-\f1p}\D
e^{t\D}v_0\bigr\|_{L^p_t(L^p)}\leq C\|\D
v_0\|_{\dB^{-1}_{p,p}}\leq C\|v_0\|_{W^{1,p}},\\
&&\bigl\|t^{1-\f1p}\D e^{t\D}v_0\bigr\|_{L^p_t(L^p)}\leq C\|\D
v_0\|_{\dB^{-2}_{p,p}}\leq C\|v_0\|_{W^{1,p}}, \eeno so that for any
$r\in ]1,2[$ and $t>0,$ one has \beq\label{S2eq2} \begin{split}
\bigl\|\D e^{t\D}v_0\bigr\|_{L^{r}_t(L^p)}\leq &
\bigl\|t^{\f12-\f1p}\D e^{t\D}v_0\bigr\|_{L^p(]0,t\wedge
1[\times\R^2)}\bigl\|t^{-\bigl(\f12-\f1p\bigr)}\bigr\|_{L^{\bar{r}}(]0,t\wedge
1[)}\\
&+\bigl\|t^{1-\f1p}\D e^{t\D}v_0\bigr\|_{L^p(]t\wedge
1,t[\times\R^2)}\bigl\|t^{-\bigl(1-\f1p\bigr)}\bigr\|_{L^{\bar{r}}(]t\wedge
1,t[)}\\
\leq & C\|v_0\|_{W^{1,p}}, \end{split} \eeq where
$\bar{r}=\f{pr}{p-r},$ and we used the fact that
$\bar{r}\bigl(\f12-\f1p\bigr)<1,$ $\bar{r}\bigl(1-\f1p\bigr)>1$
for any $r\in ]1,2[.$

Along the same line, one has \beno \bigl\|t^{\f12-\f1p}\na
e^{t\D}v_0\bigr\|_{L^p_t(L^p)}\leq C\|\na v_0\|_{\dB^{-1}_{p,p}}\leq
C\|v_0\|_{L^{p}},\eeno so that we have \beno
\begin{split}
\bigl\|\na e^{t\D}v_0\|_{L^{s}_t(L^p)}\leq & \bigl\|t^{\f12-\f1p}\na
e^{t\D}v_0\bigr\|_{L^p(]t\wedge
1,t[\times\R^2)}\bigl\|t^{-\bigl(\f12-\f1p\bigr)}\bigr\|_{L^{\bar{s}}(]t\wedge
1,t[)}\\
&+\bigl\|\na e^{t\D}v_0\bigr\|_{L^\infty(]0,t\wedge 1[; L^p)} \leq
C\bigl(\|v_0\|_{L^p}+\|\na v_0\|_{L^p}\bigr), \end{split} \eeno
where $\bar{s}=\f{ps}{p-s}$  and $\bar{s}\bigl(\f12-\f1p\bigr)>1$
for any $s\in ]2,p[.$ This together with the fact: $\bigl\|\na
e^{t\D}v_0\bigr\|_{L^\infty(\R^+; L^p)}\lesssim \|\na v_0\|_{L^p},$
ensures that \beq\label{S2eq3} \bigl\|\na
e^{t\D}v_0\|_{L^{s}_t(L^p)}\leq C\|v_0\|_{W^{1,p}}\quad\forall\ s\in
]2,\infty].\eeq

Finally it follows from Lemma \ref{lem2.1} that $\na v_0\in
\dB^{-\f1p}_{2p,p}\cap \dB^{-1-\f1p}_{2p,p}\hookrightarrow
\dB^{-\f1p}_{2p,2p}\cap \dB^{-1-\f1p}_{2p,2p}.$ Then we deduce from
Theorem 2.34 of \cite{BCD} that \beq\label{S2eq4} \begin{split}
\bigl\|\na e^{t\D}v_0\bigr\|_{L^q_t(L^{2p})}\leq & \bigl\|\na
e^{t\D}v_0\bigr\|_{L^{2p}(]0,t\wedge 1[; L^{2p})}+\bigl\|t^{\f12}\na
e^{t\D}v_0\bigr\|_{L^{2p}(]t\wedge 1,t[;
L^{2p})}\bigl\|t^{-\f12}\bigr\|_{L^{\bar{q}}(]t\wedge
1,t[)}\\
\leq & C\Bigl(\|\na v_0\|_{\dB^{-\f1p}_{2p,2p}}+\|\na
v_0\|_{\dB^{-1-\f1p}_{2p,2p}}\Bigr)\leq C\|v_0\|_{W^{1,p}},
\end{split} \eeq $\bar{q}=\f{2pq}{2p-q},$ and we
used the fact that $\f{\bar{q}}2>1$ for any $q$ satisfying $q \in
\bigl]\f{2p}{p+1}, 2p\bigr[.$

By summing up \eqref{S2eq2}, \eqref{S2eq3} and \eqref{S2eq4}, we
achieve \eqref{S2eq0}.
\end{proof}

\begin{prop}\label{S2prop2}
{\sl Let $p\in ]2,4[, v_0\in W^{1,p},$ and $\r_0$ satisfy \beq
\label{S2eq4a} \r_0-1\in L^2,\quad m\leq \r_0\leq M. \eeq Let $(\r,
v)$ be a smooth  enough solution of \eqref{inNS} on $[0,T].$ Then
there exists a positive constant, $C,$ which depends on $m, M,
\|\r_0-1\|_{L^2}$ and $\|v_0\|_{W^{1,p}}$ so that \beq \label{S2eq9}
\|({v}-v_L)(t)\|_{L^2}^2+\|\na ({v}-v_L)\|_{L^2_t(L^2)}^2\leq
C\bigl(1+{t}^2\bigr)^{\bigl(1-\f2p\bigr)}\exp\bigl(C(1+t)^{\f12}\bigr)\eqdefa
\h(t) \eeq for any $t\leq T.$}
\end{prop}
\begin{proof} We first deduce from \eqref{inNS} and \eqref{S2eq4a}
that \beq \label{S2eq4b}
\|\r(t,\cdot)-1\|_{L^2}=\|\r_0-1\|_{L^2},\quad m\leq \r(t,x)\leq M.
\eeq

 Let us set $\bar{v}(t)\eqdefa v(t)-v_L(t).$ Then  in view of
\eqref{inNS}, we write \beq\label{S2eq4c}
\begin{split}
\r\p_t\bar{v}+\r v\cdot\na\bar{v}-\D\bar{v}+\na\pi=&-\r\p_tv_L-\D
v_L-\r v\cdot\na v_L\\
=&-(\r-1)\p_tv_L-\r v\cdot\na v_L.  \end{split} \eeq Taking $L^2$
inner product of the above equation with $\bar v$ and using the
transport equation of \eqref{inNS}, we get \beq \label{S2eq5}
\f12\f{d}{dt}\|\sqrt{\r}\bar v(t)\|_{L^2}^2+\|\na \bar
v(t)\|_{L^2}^2=-\bigl((\r-1)\p_tv_L\ |\ \bar{v}\bigr)_{L^2}-\bigl(\r
v\cdot\na v_L\ |\ \bar{v}\bigr)_{L^2}. \eeq By applying the 2-D
interpolation inequality that \beq\label{S2eq6} \|f\|_{L^q}\leq
C\|f\|_{L^2}^{\f2q}\|\na f\|_{L^2}^{1-\f2q}\quad \mbox{for}\
\forall\ q\in ]2,\infty[, \eeq and Young's inequality, we obtain
\beno
\begin{split}
\bigl|\bigl((\r-1)\p_tv_L\ |\ \bar{v}\bigr)_{L^2}\bigr|\leq
&\|\r-1\|_{L^2}\|\D v_L\|_{L^p}\|\bar{v}\|_{L^{\f{2p}{p-2}}}\\
\leq &C\Bigl(\|\r_0-1\|_{L^2}^2\|\D
v_L\|_{L^p}^{2-\f{r(p-2)}p}\Bigr)^{\f12}\\
&\qquad\times \Bigl(\|\D
v_L\|_{L^p}^r\|\bar{v}\|_{L^2}^2\Bigr)^{\f{p-2}{2p}}\Bigl(\|\na
\bar{v}\|_{L^2}^2\Bigr)^{\f1p}\\
\leq &\f14\|\na \bar{v}\|_{L^2}^2+C\Bigl(\|\r_0-1\|_{L^2}^2\|\D
v_L\|_{L^p}^{2-\f{r(p-2)}p}+\|\D
v_L\|_{L^p}^r\|\bar{v}\|_{L^2}^2\Bigr),
\end{split}
\eeno for any $r\in ]1,2[.$

Whereas it is easy to observe that \beno
\begin{split}
\bigl|\bigl(\r \bar{v}\cdot\na v_L\ |\ \bar{v}\bigr)_{L^2}\bigr|\leq
& C\|\na v_L\|_{L^p}\|\bar{v}\|_{L^{\f{2p}{p-1}}}^2\\
\leq &C\|\na v_L\|_{L^p}\|\bar
v\|_{L^2}^{2\bigl(1-\f1p\bigr)}\|\na\bar v\|_{L^2}^{\f2p}\\
\leq & \f14\|\na \bar{v}\|_{L^2}^2+C \|\na
v_L\|_{L^p}^{\f{p}{p-1}}\|\bar v\|_{L^2}^2.
\end{split}
\eeno

In order the deal with the last term in \eqref{S2eq5}, we need to
use the claim that \beq\label{S2eq7} \|v_L\otimes
v_L\|_{L^{\f{p}2}_t(\dB^1_{2,\f{p}2})}\leq
C\|v_0\|_{L^p}^2\quad\mbox{for}\ \  \forall\ p\in ]2,4[. \eeq We
admit this inequality for the time being and continue to handle the
last term in \eqref{S2eq5}. Since $p\in ]2,4[,$ we have \beno
\begin{split}
\bigl|\bigl(\r v_L\cdot\na v_L\ |\ \bar{v}\bigr)_{L^2}\bigr|\leq
&\|v_L\otimes v_L\|_{\dH^1}\|\bar v\|_{L^2}\\
\leq &C\Bigl(\|v_L\otimes v_L\|_{\dB^1_{2,\f{p}2}}^{\f{p}2}\|\bar
v\|_{L^2}^2+\|v_L\otimes v_L\|_{\dB^1_{2,\f{p}2}}^{2-\f{p}2}\Bigr).
\end{split}
\eeno Substituting the above estimates into \eqref{S2eq5} gives rise
to \beq \label{S2eq8}
\f{d}{dt}\|\sqrt{\r}\bar{v}(t)\|_{L^2}^2+\|\na\bar{v}\|_{L^2}^2\leq
C\bigl(f_1(t)\|\sqrt{\r}\bar{v}(t)\|_{L^2}^2+f_2(t)\bigr), \eeq
where $f_1(t), f_2(t)$ are defined by \beno &&f_1(t)=\|\D
v_L(t)\|_{L^p}^r+\|\na v_L(t)\|_{L^p}^{\f{p}{p-1}}+\|v_L\otimes
v_L\|_{\dB^1_{2,\f{p}2}}^{\f{p}2},\\
&&f_2(t)=\|\D v_L(t)\|_{L^p}^{2-\f{r(p-2)}p}+\|v_L\otimes
v_L\|_{\dB^1_{2,\f{p}2}}^{2-\f{p}2}\quad\mbox{for}\ r\in ]1,2[.
\eeno By applying Lemma \ref{S2lem1} and \eqref{S2eq7}, we achieve
\beno
\begin{split} \int_0^tf_1(t')\,dt'\leq &
C\Bigl(\|v_0\|_{W^{1,p}}^r+\Bigl(\int_0^t\|\na
v_L(t')\|_{L^p}^{\f{2p}{p-1}}\,dt'\Bigr)^{\f12}\sqrt{t}+\|v_0\|_{L^p}^p\Bigr)\\
\leq &
C\Bigl(\|v_0\|_{W^{1,p}}^r+\|v_0\|_{W^{1,p}}^{\f{p}{p-1}}\sqrt{t}+\|v_0\|_{L^p}^p\Bigr),
\end{split}
\eeno and since $p\in ]2,4[, r\in ]1,2[,$ we have
$2-\f{r(p-2)}{p}\in ]1,2[$ and \beno
\begin{split} \int_0^tf_2(t')\,dt'\leq &
C\|v_0\|_{W^{1,p}}^{2-\f{r(p-2)}{p}}+C\Bigl(\int_0^t\|v_L\otimes
v_L(t')\|_{\dB^1_{2,\f{p}2}}^{\f{p}2}\,dt'\Bigr)^{\f{4-p}{p}}t^{2\bigl(1-\f2p\bigr)}\\
\leq  &
C\Bigl(\|v_0\|_{W^{1,p}}^{2-\f{r(p-2)}{p}}+\|v_0\|_{L^p}^{8-2p}t^{2\bigl(1-\f2p\bigr)}\Bigr).
\end{split}
\eeno  Therefore by applying Gronwall's inequality to \eqref{S2eq8},
we obtain \eqref{S2eq9}.
\end{proof}

Let us now turn to the proof of \eqref{S2eq7}.

\begin{proof}[Proof of \eqref{S2eq7}] We first get, by applying
Bony's decomposition \eqref{bony} in the homogeneous context to
$v_L\otimes v_L,$ that \beno v_L\otimes v_L=2\dot T_{v_L}v_L+\dot
R(v_L,v_L). \eeno Due to $v_0\in \dB^0_{p,p}$ and $p\in ]2,4[,$ we
deduce from Lemma \ref{lem2.1} that \beno\begin{split} &\|\dot
S_{j'-1}v_L\|_{L^{\f{2p}{p-2}}}\lesssim \sum_{\ell\leq
j'-2}2^{\ell\bigl(\f4p-1\bigr)}c_{\ell,p}\|v_0\|_{\dB^0_{p,p}}\lesssim
c_{j',p}2^{j'\bigl(\f4p-1\bigr)}\|v_0\|_{L^p} \andf\\
&\|\dot\D_{j'}v_L\|_{L^{\f{p}2}_t(L^p)}\lesssim
\bigl\|e^{-c2^{2j'}t}\dot\D_{j'}v_0\|_{L^{\f{p}2}_t(L^p)}\lesssim
c_{j',p}2^{-\f{4j'}p}\|v_0\|_{L^p}, \end{split} \eeno from which, we
infer \beno
\begin{split}
\bigl\|\dot\D_j(\dot T_{v_L}v_L)\bigr\|_{L^{\f{p}2}_t(L^2)}\lesssim
&\sum_{|j'-j|\leq
4}\|\dot S_{j'-1}v_L\|_{L^\infty_t(L^{\f{2p}{p-2}})}\|\dot\D_{j'}v_L\|_{L^{\f{p}2}_t(L^p)}\\
\lesssim &\sum_{|j'-j|\leq 4}
c_{j',\f{p}2}2^{-j'}\|v_0\|_{L^p}^2\lesssim
c_{j,\f{p}2}2^{-j}\|v_0\|_{L^p}^2,
\end{split}
\eeno where $\bigl(c_{j,p}\bigr)_{j\in\Z}$ denotes a unit  generic
element of $\ell^p(\Z).$

 Whereas by applying Lemma \ref{lem2.1} once
again, one has \beno
\begin{split}
\bigl\|\dot\D_j\dot R(v_L,v_L)\|_{L^{\f{p}2}_t(L^2)}\lesssim &
2^{j'\bigl(\f4p-1\bigr)}\sum_{j'\geq
j-3}\|\dot\D_{j'}v_L\|_{L^{\f{p}2}_t(L^p)}\|\wt{\dot\D}_{j'}v_L\|_{L^\infty_t(L^p)}\\
\lesssim & 2^{j'\bigl(\f4p-1\bigr)}\sum_{j'\geq j-3}
c_{j',\f{p}2}2^{-\f{4j'}p}\|v_0\|_{L^p}^2\\
\lesssim & c_{j,\f{p}2}2^{-j}\|v_0\|_{L^p}^2.
\end{split}
\eeno This proves \eqref{S2eq7}.
\end{proof}

\smallbreak

\begin{lem}\label{S2lem2}
{\sl For any  $r\in ]1,2[$ and $ p\in ]1,\infty[,$ one has
\begin{equation}\label{S2eq14}
\Bigl\| \int^t_0 \nabla e^{(t-t')\Delta} f(t')\,dt'
\Bigr\|_{L^{\frac{2r}{2-r}}_T (L^p)}+\Bigl\| \int^t_0 \nabla
e^{(t-t')\Delta} f(t')\,dt' \Bigr\|_{L^{q}_T (L^{2p})} \lesssim
\|f\|_{L^r_T(L^p)},
\end{equation} for $q$ given by $\f1q=\f1r-\f12\bigl(1-\f1p\bigr).$}
\end{lem}

\begin{proof}
As a matter of fact, note that
\begin{align*}
\nabla e^{(t-t')\Delta}f(t') &= \frac{\sqrt\pi}{(4\pi (t-t'))^{3/2}}
\int_{\R^2} \frac{x-y}{2(t-t')^{1/2}} \exp\Bigl \{
-\frac{|x-y|^2}{4(t-t')} \Bigr\} \, f(t',y)\,dy
\\
&\eqdefa\frac{\sqrt\pi}{(4\pi (t-t'))^{3/2}} K\Bigl(
\frac{\cdot}{2(t-t')^{1/2}} \Bigr) \ast \, f(t',x) .
\end{align*}
Then applying Young's inequality and then Hardy's inequality leads
to
\begin{align*}
\Bigl\| \int^t_0 \nabla e^{(t-t')\Delta} f(t')\,dt'
\Bigr\|_{L^{\frac{2r}{2-r}}_T (L^p)} &\lesssim \Bigl\| \int^t_0
(t-t')^{-\frac 12} \|f(t')\|_{L^p}\,dt'
\Bigr\|_{L^{\frac{2r}{2-r}}_T }
\\
&\lesssim \|f\|_{L^r_T(L^p)},
\end{align*}
and
\begin{align*}
\Bigl\| \int^t_0 \nabla e^{(t-t')\Delta} f(t')\,dt' \Bigr\|_{L^{q}_T
(L^{2p})} &\lesssim \Bigl\| \int^t_0 (t-t')^{-\frac
12\bigl(1+\f1p\bigr)} \|f(t')\|_{L^p}\,dt' \Bigr\|_{L^{q}_T }
\\
&\lesssim \|f\|_{L^r_T(L^p)}.
\end{align*}
This proves \eqref{S2eq14}. \end{proof}

\begin{prop}\label{S2prop1}
{\sl Let $r, \s$ be  any numbers so that $r\in ]1,2[$ and $\s\in
\bigl]\f{2p}{p+2}, p\bigr[.$ Then under the assumptions of
Proposition \ref{S2prop2}, if we assume moreover that
$\|a_0\|_{L^\infty}$ is sufficiently small,  there holds
\beq\label{S4eq2}
\begin{split}
J_0\eqdefa 1+\|(\d_t v,\Delta v,\nabla\pi)\|_{L^{r}_t(L^p)}&
+\|\nabla v\|_{L^\s_t(L^\infty)\cap L^{\f{2r}{2-r}}_t(L^p)}
\\
&+\|v\|_{L^{{\f{2r}{2-r}}}_t(L^\infty)} +\|X\|_{L^\infty_t(W^{1,p})}
\leq \cH_0(t).
\end{split}
\eeq}
\end{prop}

\begin{proof} We divide the proof of this proposition into the
following steps:

\no$\bullet$ \underline{The estimate of $\|(\d_t v,\Delta
v,\nabla\pi)\|_{L^{r}_t(L^p)}$}

We first rewrite the momentum equation of \eqref{INS} as
$$
v(t)=e^{t\Delta}v_0 +\int^t_0 e^{(t-t')\Delta}\Bigl( -v\cdot\nabla v
+a\Delta v-(1+a)\nabla\pi \Bigr)(t')\,dt'.
$$
For any $\la>0,$ we set \beq \label{S2eq10} f(t)\eqdefa
\|v(t)\|_{L^{2p}}^{\f{2p}{p-1}}\andf v_\la(t)\eqdefa
v(t)\exp\bigl(-\la\int_0^tf(t')\,dt'\bigr). \eeq Then $v_\la$ solve
\beq \label{S2eq11} v_\la(t)=e^{t\Delta}v_{0,\la} +\int^t_0
e^{-\la\int_{t'}^tf(t'')\,dt''} e^{(t-t')\Delta}\Bigl( -v\cdot\nabla
v_\la +a\Delta v_\la-(1+a)\nabla\pi_\la \Bigr)(t')\,dt'. \eeq
Whereas by taking space divergence to the momentum equation of
\eqref{INS}, one has \beq\label{S2eq11a} \D\pi_\la=-\dive(v\cdot\na
v_\la)+\dive\bigl(a(\D v_\la-\na\pi_\la)\bigr), \eeq from which and
$\|a_0\|_{L^\infty}$ is sufficiently small, we infer
\beq\label{S2eq12} \|\na\pi_\la(t)\|_{L^p}\leq C\bigl(\|v\cdot\na
v_\la(t)\|_{L^p}+\|a_0\|_{L^\infty}\|\D v_\la(t)\|_{L^p}\bigr). \eeq

In view of \eqref{S2eq11}, we get, by applying maximal regularity
estimate to heat semigroup (see Theorem 7.3 of \cite{LPG} for
instance), and the Estimates \eqref{S2eq0},
 \eqref{S2eq14} and \eqref{S2eq12},
that \beno
\begin{split}
\|\D v_\la&\|_{L^r_t(L^p)}+\|\bigl(\p_t
v\bigr)_\la\|_{L^r_t(L^p)}+\|\na
v_\la\|_{L^{\f{2r}{2-r}}_t(L^p)}+\|\na v_\la\|_{L^q_t(L^{2p})}\\
\leq &\|\D v_L \|_{L^r_t(L^p)}+\|\p_t v_L \|_{L^r_t(L^p)}+\|\na v_L
\|_{L^{\f{2r}{2-r}}_t(L^p)}+\|\na v_L
\|_{L^{q}_t(L^{2p})}\\
&+C\biggl(\int_0^te^{-\la r\int_{t'}^tf(t'')\,dt''}\Bigl(\|v\cdot\na
v_\la(t')\|_{L^p}^r+\|a(t')\|_{L^\infty}^r\|\D
v_\la(t')\|_{L^p}^r\\
&\qquad\qquad\qquad\qquad\qquad\qquad\qquad\quad+\bigl(1+\|a(t')\|_{L^\infty}^r\bigr)\|\na\pi(t')\|_{L^p}^r\Bigr)\,dt'\biggr)^{\f1r}\\
\leq& C\biggl(\|v_0\|_{W^{1,p}}+\|a_0\|_{L^\infty}\|\D
v_\la\|_{L^r_t(L^p)}\\
&\qquad\qquad\qquad\qquad+\Bigl(\int_0^te^{-\la
r\int_{t'}^tf(t'')\,dt''}\|v\cdot\na
v_\la(t')\|_{L^p}^r\,dt'\Bigr)^{\f1r}\biggr),
\end{split}
\eeno where  $\f1q=\f1r-\f12\bigl(1-\f1p\bigr)$ and thus
$q\in\bigl]\f{2p}{p+1},2p\bigr[.$ Hence whenever
$\|a_0\|_{L^\infty}$ is sufficiently small, we obtain \beno
\begin{split}
\|\D v_\la\|_{L^r_t(L^p)}+&\|\bigl(\p_t
v\bigr)_\la\|_{L^r_t(L^p)}+\|\na
v_\la\|_{L^{\f{2r}{2-r}}_t(L^p)}+\|\na v_\la\|_{L^q_t(L^{2p})}\\
\leq
&C\biggl(\|v_0\|_{W^{1,p}} +\Bigl(\int_0^te^{-\la
r\int_{t'}^tf(t'')\,dt''}\|v\cdot\na
v_\la(t')\|_{L^p}^r\,dt'\Bigr)^{\f1r}\biggr).
\end{split}\eeno
Whereas due to $\f1r=\f1q+\f12\bigl(1-\f1p\bigr),$  it is easy to
observe that \beno
\begin{split}
\Bigl(\int_0^t&e^{-\la r\int_{t'}^tf(t'')\,dt''}\|v\cdot\na
v_\la(t')\|_{L^p}^r\,dt'\Bigr)^{\f1r}\\
\leq
&\Bigl(\int_0^te^{-\f{2p\la}{p-1}
\int_{t'}^tf(t'')\,dt''}\|v(t')\|_{L^{2p}}^{\f{2p}{p-1}}\,dt'\Bigr)^{\f{p-1}{2p}}\|\na
v_\la\|_{L^q_t(L^{2p})}\\
\leq &\f{C}{\la ^{\f{p-1}{2p}}}\|\na v_\la\|_{L^q_t(L^{2p})}.
\end{split}
\eeno  As a consequence, we obtain \beq \label {S2eq13}
\begin{split}
\|\D v_\la\|_{L^r_t(L^p)}+\|\bigl(\p_t
v\bigr)_\la\|_{L^r_t(L^p)}+\|\na
v_\la\|_{L^{\f{2r}{2-r}}_t(L^p)}+&\|\na
v_\la\|_{L^q_t(L^{2p})}\\
\leq &C\Bigl(\|v_0\|_{W^{1,p}} +\f{1}{\la ^{\f{p-1}{2p}}}\|\na
v_\la\|_{L^q_t(L^{2p})}\Bigr).
\end{split}\eeq

Then in view of \eqref{S2eq10}, by taking $\la$ to be sufficiently
large in \eqref{S2eq13}, we achieve \beno
\begin{split} & e^{-\la\int_0^tf(t')\,dt'}\bigl(
\|\D v\|_{L^r_t(L^p)}+\|\p_t v\|_{L^r_t(L^p)}+\|\na
v\|_{L^{\f{2r}{2-r}}_t(L^p)}+\|\na
v\|_{L^q_t(L^{2p})}\bigr)\\
&\qquad\qquad \leq \|\D v_\la\|_{L^r_t(L^p)}+\|\bigl(\p_t
v)_\la\|_{L^r_t(L^p)}+\|\na v_\la\|_{L^{\f{2r}{2-r}}_t(L^p)}+\|\na
v_\la\|_{L^q_t(L^{2p})}\leq C\|v_0\|_{W^{1,p}}.
\end{split}
\eeno While it follows from  \eqref{S2eq9} and $v=\bar{v}+v_L$
 that \beno
\begin{split}
\|v\|_{L^{\f{2p}{p-1}}_t(L^{2p})}^{\f{2p}{p-1}}\leq
&\int_0^t\Bigl(\|\bar{v}(t')\|_{L^{2p}}+\|v_L(t')\|_{L^{2p}}\Bigr)^{\f{2p}{p-1}}\,dt'\\
\leq
&C\Bigl(\|\bar{v}\|_{L^\infty_t(L^2)}^{\f2{p-1}}\|\na\bar{v}\|_{L^2_t(L^2)}^2+t\|v_L\|_{L^\infty_t(L^p)}^2\|\na
v_L\|_{L^\infty_t(L^p)}^{\f2{p-1}}\Bigr)\\
\leq &C\Bigl(\h^{\f{2p}{p-1}}(t)+t\|v_0\|_{L^p}^2\|\na
v_0\|_{L^p}^{\f2{p-1}}\Bigr)\leq \cH_0(t),
\end{split}
\eeno so that we obtain \beq\label{S2eq15}\begin{split} \|\D
v\|_{L^r_t(L^p)}+\|\p_t v\|_{L^r_t(L^p)}+&\|\na
v\|_{L^{\f{2r}{2-r}}_t(L^p)}+\|\na v\|_{L^q_t(L^{2p})}\\
\leq& Ce^{\la \int_0^tf(t')\,dt'}\|v_0\|_{W^{1,p}}\leq Ce^{\la
\cH_0(t)}\|v_0\|_{W^{1,p}} \leq \cH_0(t). \end{split} \eeq This
together with \eqref{S2eq12} ensures that \beq
\label{S2eq15a}\begin{split} \|\na\pi\|_{L^r_t(L^p)}\leq
C\Bigl(\|v\|_{L^{\f{2p}{p-1}}_t(L^{2p})}\|\na
v\|_{L^q_t(L^{2p})}+\|\D v\|_{L^r_t(L^p)}\Bigr)\leq \cH_0(t).
\end{split}
\eeq Moreover, there holds \beq \label{S2eq15b}\begin{split}
&\|\na v\|_{L^{\f{2pr}{2(p+r)-pr}}_t(L^\infty)}\leq C\|\na
v\|_{L^{\f{2r}{2-r}}_t(L^p)}^{1-\f2p}\|\D
v\|_{L^r_t(L^p)}^{\f2p}\leq \cH_0(t)\quad\mbox{for}\ \ \forall \
r\in ]1,2[.
\end{split}
\eeq It is easy to observe that when $r$ varies from $1$ to $2,$ we
have $\s=\f{2pr}{2(p+r)-pr}\in \bigl]\f{2p}{2+p}, p\bigr[.$

\no$\bullet$ \underline{The estimate of
$\|v\|_{L^{\f{2r}{2-r}}_t(L^\infty)}$}

In view of \eqref{S2eq4c}, we get by a  similar  derivation of
\eqref{S2eq11} that \beno \bar{v}_\la(t)=-\int^t_0
e^{-\la\int_{t'}^tf(t'')\,dt''}
e^{(t-t')\Delta}\Bbb{P}\Bigl((\r-1)(\p_t\bar{v}+\p_tv_{L})_\la+\r
v\cdot\nabla \bar{v}_\la +\r v\cdot\na v_{L,\la} \Bigr)(t')\,dt',
\eeno where $\Bbb{P}\eqdefa Id-\na\D^{-1}\dive$ denotes the Leray
projection operator to the divergence free vector space. Then
applying the maximal regularity estimate to heat semi-group and
Lemma \ref{S2lem2} yields \beno
\begin{split}
\|\D \bar{v}_\la&\|_{L^r_t(L^p)}+\|\bigl(\p_t
\bar{v}\bigr)_\la\|_{L^r_t(L^p)}+\|\na
\bar{v}_\la\|_{L^{\f{2r}{2-r}}_t(L^p)}\\
\leq &C\bigl(\|(\r-1)(\p_t\bar{v})_\la\|_{L^r_t(L^p)}+\|(\r-1)\p_t
v_L
\|_{L^r_t(L^p)}\bigr)\\
&+C\biggl(\int_0^te^{-\la r\int_{t'}^tf(t'')\,dt''}\Bigl(\|v\cdot\na
\bar v_\la(t')\|_{L^p}^r+\| v\cdot\na
v_{L,\la}(t')\|_{L^p}^r\Bigr)\,dt'\biggr)^{\f1r}.
\end{split}
\eeno
By using Sobolev embedding Theorem, one has
\beno \|v\cdot\na
\ov v_\la(t')\|_{L^p}\leq
C\|v(t')\|_{L^{2p}}\|\na\bar{v}_\la(t')\|_{L^p}^{1-\f1p}\|\D
\bar{v}_\la(t')\|_{L^p}^{\f1p}, \eeno so that thanks to
\eqref{S2eq10},  we get, by applying H\"older's inequality, that
\beno
\begin{split}
\Bigl(\int_0^t&e^{-\la r\int_{t'}^tf(t'')\,dt''}\|v\cdot\na
\bar{v}_\la(t')\|_{L^p}^r\,dt'\Bigr)^{\f1r}\\
\leq &C\Bigl(\int_0^t
e^{-\f{\la pr}{p-1} \int_{t'}^tf(t'')\,dt''}\|v(t')\|_{L^{2p}}^{\f{
pr}{p-1}}\|\na\bar{v}_\la(t')\|_{L^p}^r\,dt'\Bigr)^{\f1r}+\f14\|\D
\bar{v}_\la(t')\|_{L^r_t(L^p)}\\
\leq &C \Bigl(\int_0^t e^{-\frac{2p\lambda}{p-1}
\int_{t'}^tf(t'')\,dt''}
\|v(t')\|_{L^{2p}}^{\f{2p}{p-1}}\,dt'\Bigr)^{\f12}
\|\na\bar{v}_\la(t')\|_{L^{\f{2r}{2-r}}_t(L^p)}
+\f14\|\D\bar{v}_\la(t')\|_{L^r_t(L^p)}\\
\leq
&\f{C}{\la^{\f 12}}\|\na\bar{v}_\la\|_{L^{\f{2r}{2-r}}_t(L^p)}+\f14\|\D
\bar{v}_\la\|_{L^r_t(L^p)}.
\end{split}
\eeno Along the same line, we have \beno \Bigl(\int_0^te^{-\la
r\int_{t'}^tf(t'')\,dt''}\|v\cdot\na
v_{L,\la}(t')\|_{L^p}^r\,dt'\Bigr)^{\f1r}\leq
\f{C}{\la^{\f12}}\|\na{v}_L\|_{L^{\f{2r}{2-r}}_t(L^p)}+\f14\|\D
{v}_L\|_{L^r_t(L^p)}. \eeno As a consequence, we obtain \beno
\begin{split}
\|\D &\bar{v}_\la\|_{L^r_t(L^p)}+\|\bigl(\p_t
\bar{v}\bigr)_\la\|_{L^r_t(L^p)}+\|\na
\bar{v}_\la\|_{L^{\f{2r}{2-r}}_t(L^p)}\\
\leq &C\Bigl(\|\r_0-1\|_{L^\infty}\|\bigl(\p_t
\bar{v}\bigr)_\la\|_{L^r_t(L^p)}
+\frac{1}{\lambda^{\frac 12}}\|\na{v}_L\|_{L^{\f{2r}{2-r}}_t(L^p)}+\|\D
{v}_L\|_{L^r_t(L^p)}+\f{1}{\la^{\f12}}\|\na\bar{v}_\la\|_{L^{\f{2r}{2-r}}_t(L^p)}\Bigr).
\end{split}
\eeno Therefore, by virtue  of \eqref{S2eq0}, by taking $\la$ large
enough and $\|\r_0-1\|_{L^\infty}$ to be sufficiently small, we
achieve \beno \|\D \bar{v}_\la\|_{L^r_t(L^p)}+\|\bigl(\p_t
\bar{v}\bigr)_\la\|_{L^r_t(L^p)}+\|\na
\bar{v}_\la\|_{L^{\f{2r}{2-r}}_t(L^p)}\leq C\|v_0\|_{W^{1,p}}. \eeno
This in particular implies \beno \|\na
\bar{v}\|_{L^{\f{2r}{2-r}}_t(L^p)}\leq C
e^{\la\int_0^tf(t')\,dt'}\|v_0\|_{W^{1,p}}\leq C\cH_0(t),\eeno from
which, \eqref{S2eq9} and the interpolation inequality that \beno
\|f\|_{L^\infty}\leq C\|f\|_{L^2}^{\f{p-2}{2(p-1)}}\|\na
f\|_{L^p}^{\f{p}{2(p-1)}},\eeno  we infer \beq\label{S2eq15c} \|\bar
v\|_{L^{\f{2r}{2-r}}_t(L^\infty)}\leq C \|\bar
v\|_{L^{\f{2r}{2-r}}_t(L^2)}^{\f{p-2}{2(p-1)}}\|\na \bar
v\|_{L^{\f{2r}{2-r}}_t(L^p)}^{\f{p}{2(p-1)}}\leq C\cH_0(t). \eeq
While it is easy to notice that \beno
\|v_L\|_{L^\infty_t(L^\infty)}\leq
C\|v_L\|_{L^\infty_t(L^p)}^{1-\f2p}\|\na
v_L\|_{L^\infty_t(L^p)}^{\f2p}\leq C\|v_0\|_{W^{1,p}}, \eeno this
together with \eqref{S2eq15c} ensures that \beq\label{S2eq15d}
\|v\|_{L^{\f{2r}{2-r}}_t(L^\infty)}\leq C\cH_0(t). \eeq

\no$\bullet$ \underline{The estimate of
$\|X\|_{L^{\infty}_t(W^{1,p})}$}

Finally, as $X$ satisfies  \eqref{S1eq2}, it follows from standard
theory on transport equation and \eqref{S2eq15b} that \beq
\label{S2eq17} \|X(t)\|_{L^p\cap L^\infty}\leq \|X_0\|_{L^p\cap
L^\infty}\exp\Bigl(\int_0^t\|\na v(t')\|_{L^\infty}\,dt'\Bigr)\leq
\cH_0(t). \eeq By taking $\p_i, i=1,2,$ to \eqref{S1eq2}, we get
$$
\d_t \d_i X+v\cdot \nabla \d_i X+\d_i v\cdot\nabla X =\d_i
X\cdot\nabla v+X\cdot\na \d_i v,
$$
from which, we get, by applying standard $L^p$ type energy estimate,
that \beno \f{d}{dt}\|\na X(t)\|_{L^p}\leq 2\|\na
v(t)\|_{L^\infty}\|\na X(t)\|_{L^p}+\|X(t)\|_{L^\infty}\|\na^2
v\|_{L^p}. \eeno Applying Gronwall's inequality and using
\eqref{S2eq15b}, \eqref{S2eq17} yields
\beq \label{S2eq18}
\begin{split} \|\na X(t)\|_{L^p}\leq &\Bigl(\|\na
X_0\|_{L^p}+\|X\|_{L^\infty_t(L^\infty)}\|\na^2
v\|_{L^1_t(L^p)}\Bigr)\exp\Bigl(2\int_0^t\|\na
v(t')\|_{L^\infty}\,dt'\Bigr)\\
\leq &\Bigl(\|\na X_0\|_{L^p}+\|X_0\|_{L^\infty}\|\na^2
v\|_{L^1_t(L^p)}\Bigr)\exp\Bigl(3\int_0^t\|\na
v(t')\|_{L^\infty}\,dt'\Bigr)\leq \cH_0(t).
\end{split} \eeq
This completes the proof of \eqref{S4eq2} and hence the proposition.
\end{proof}

\setcounter{equation}{0}
\section{The estimates of $\D\p_Xv, \p_t\p_Xv,
\na\p_Xv$}\label{S5}

In this section, we investigate the estimates of $\D\p_Xv,
\p_t\p_Xv, \na\p_Xv$ under the striated regularity  assumption of
the initial data:
\begin{equation}\label{initial:striated}
\d_{X_0}^\ell a_0\in L^\infty, \,\d_{X_0}^\ell v_0\in
W^{1-\frac{\ell}{k}\e,p}, \,\d_{X_0}^{\ell-1} X_0\in W^{2,p},
\quad\e\in \left]0,1\right[, \, \forall \ell=1,\cdots,k.
\end{equation}

\begin{lem}\label{S3lem1}
{\sl Let $r_\ell\in \bigl]1,\f{2k}{k+\ell\e}\bigr[$ and $s_\ell\in
\bigl]2,\f{2k}{\ell\e}\bigr[.$ Then under the assumption of
\eqref{initial:striated}, one has \beq\label{S3eq0} \bigl\|\D
e^{t\D}\p_{X_0}^\ell v_0\bigr\|_{L^{r_\ell}_t(L^p)}+\bigl\|\na
e^{t\D}\p_{X_0}^\ell v_0\|_{L^{s_\ell}_t(L^p)}\leq C\|\p_{X_0}^\ell
v_0\|_{W^{1-\f\ell{k}\e,p}},
\quad \forall \ell=1,\cdots,k. \eeq}
\end{lem}
\begin{proof}
Since $\p_{X_0}^\ell v_0\in W^{1-\f{\ell}k\e,p}\hookrightarrow
\dB^0_{p,p}\cap \dB^{1-\f{\ell}k\e}_{p,p},$ so that
$\na\p_{X_0}^\ell v_0\in \dB^{-1}_{p,p}\cap
\dB^{-\f{\ell}k\e}_{p,p}$ and $\D\p_{X_0}v_0\in \dB^{-2}_{p,p}\cap
\dB^{-1-\f{\ell}k\e}_{p,p}.$ And therefore, according to Theorem
2.34 of \cite{BCD}, we infer \beq\label{S3eq3}
\begin{split} \bigl\|\D
e^{t\D}&\p_{X_0}^\ell v_0\bigr\|_{L^{r_\ell}_t(L^p)}\\
\leq & \bigl\|t^{\f12\bigl(1+\f\ell{k}\e\bigr)-\f1p}\D
e^{t\D}\p_{X_0}^\ell v_0\bigr\|_{L^p(]0,t\wedge
1[\times\R^2)}\bigl\|t^{-\bigl(\f12\bigl(1+\f\ell{k}\e\bigr)-\f1p\bigr)}\bigr\|_{L^{\bar{r}_\ell}(]0,t\wedge
1[)}\\
&+\bigl\|t^{1-\f1p}\D e^{t\D}\p_{X_0}^\ell v_0\bigr\|_{L^p(]t\wedge
1,t[\times\R^2)}\bigl\|t^{-\bigl(1-\f1p\bigr)}\bigr\|_{L^{\bar{r}_\ell}(]t\wedge
1,t[)}\\
\leq & C\bigl(\|\D\p_{X_0}^\ell
v_0\|_{\dB^{-1-\f{\ell}k\e}_{p,p}}+\|\D\p_{X_0}^\ell
v_0\|_{\dB^{-2}_{p,p}}\bigr)\leq \|\p_{X_0}^\ell
v_0\|_{W^{1-\f\ell{k}\e,p}},
\end{split} \eeq where $\bar{r}_\ell=\f{r_\ell p}{p-r_\ell},$ and we used the fact
that $\bar{r}_\ell\bigl(\f12\bigl(1+\f\ell{k}\e\bigr)-\f1p\bigr)<1,$
$\bar{r}_\ell\bigl(1-\f1p\bigr)>1$ for any $r_\ell\in
\bigl]1,\f{2k}{k+\ell\e}\bigr[.$

Similarly, one has \beq\label{S3eq4}
\begin{split}
\bigl\|\na e^{t\D}\p_{X_0}^\ell v_0\|_{L^{s_\ell}_t(L^p)}\leq &
\bigl\|t^{\f{\ell}{2k}\e-\f1p}\na e^{t\D}\p_{X_0}^\ell
v_0\bigr\|_{L^p(]0,t\wedge
1[\times\R^2)}\bigl\|t^{-\bigl(\f\ell{2k}\e-\f1p\bigr)}\bigr\|_{L^{\bar{s}_\ell}(]0,t\wedge
1[)}\\
&+ \bigl\|t^{\f12-\f1p}\na e^{t\D}\p_{X_0}^\ell
v_0\bigr\|_{L^p(]t\wedge
1,t[\times\R^2)}\bigl\|t^{-\bigl(\f12-\f1p\bigr)}\bigr\|_{L^{\bar{s}_\ell}(]t\wedge
1,t[)}\\
\leq & C\bigl(\|\na\p_{X_0}^\ell
v_0\|_{\dB^{-\f{\ell}k\e}_{p,p}}+\|\na\p_{X_0}^\ell
v_0\|_{\dB^{-1}_{p,p}}\bigr)\leq \|\p_{X_0}^\ell
v_0\|_{W^{1-\f\ell{k}\e,p}},
\end{split} \eeq where $\bar{s}_\ell=\f{ps_\ell}{p-s_\ell}$ and we used the
fact that $\bar{s}_\ell\bigl(\f12-\f1p\bigr)>1,$
$\bar{s}_\ell\bigl(\f\ell{2k}\e-\f1p\bigr)<1$ for any $s_\ell\in
\bigl]2,\f{2k}{\ell\e}\bigr[.$

Combining \eqref{S3eq3} with \eqref{S3eq4}, we obtain \eqref{S3eq0}.
\end{proof}

\begin{prop}\label{S3prop1}
{\sl Let $r_1, s_1$ be any numbers with $ r_1\in
\bigl]1,\frac{2k}{k+\e}\bigr[,$ $s_1\in \bigl]2,\f{2k}\e\bigr[.$
Then under the assumptions of Proposition \ref{S2prop1} and
\eqref{initial:striated},  one has \beq\label{S3eq4a}
\begin{split}
J_1\eqdefa  &J_0 +\|\d_X a\|_{L^\infty_t(L^\infty)} +
\bigl\|\bigl(\d_t \d_X v,\Delta\d_X
v,\nabla\d_X\pi\bigr)\bigr\|_{L^{r_1}_t(L^p)}
\\
&\qquad +\|\nabla\d_X v\|_{L^{r_1}_t(L^\infty) \cap L^{s_1}_t(L^p)}
+\|\d_X v\|_{L^{s_1}_t(L^\infty)}
\\
&\qquad\qquad\quad + \|\d_t X\|_{L^{s_1}_t(W^{1,p})} + \|
X\|_{L^\infty_t(W^{2,p})} \leq  \cH_1(t).
\end{split} \eeq} \end{prop}

\begin{proof} We decompose the proof of this proposition into the
following steps:

 \no$\bullet$ \underline{The estimate of $\p_X a$ and $\D X$}

 For  $\ell=1,\cdots, k,$ it follows from the transport
 equation of \eqref{INS} and \eqref{S1eq2} that $[\p_X; D_t]=0$ for
 $D_t\eqdefa \p_t+v\cdot\na$ and thus
\beq\label{S3eq1}  \d_t\d_X^\ell a+v\cdot\nabla\d_X^\ell a=0, \eeq
which implies
\begin{equation}\label{S3eq2} \|\d_X^\ell a\|_{L^\infty_t(L^\infty)}\leq
\|\d_{X_0}^\ell a_0\|_{L^\infty} .
\end{equation}

While we get, by applying the  operator $\Delta$ to \eqref{S1eq2},
that
$$\displaylines{
\d_t \Delta X+v\cdot\nabla\Delta X+2 \sum_{i=1}^2\p_iv \cdot
\nabla\p_iX
  +\Delta v\cdot\nabla X
=\Delta \d_X v. }$$ By using the standard $L^p$ energy estimate to
the above equation and interpolation inequality that
\beq\label{S3eq2a} \|f\|_{L^\infty}\leq\|f\|_{\dot
B^{\f2p}_{p,1}}\leq C\|f\|_{L^p}^{1-\f2p}\|\na
f\|_{L^p}^{\f2p}\quad\mbox{for}\ \ \forall\  p\in ]2,\infty[, \eeq
we obtain \beno
\begin{split}
\f{d}{dt}\|\D X(t)\|_{L^p}\leq & 2\|\na
v(t)\|_{L^\infty}\|\na^2X(t)\|_{L^p}+\|\D v(t)\|_{L^p}\|\na
X(t)\|_{L^\infty}+\|\D\p_Xv(t)\|_{L^p}\\
\leq &C\Bigl(\|\na v(t)\|_{L^p}^{1-\f2p}\|\D v(t)\|_{L^p}^{\f2p}\|\D
X(t)\|_{L^p}\\
&\quad+\|\D v(t)\|_{L^p}\|\na X(t)\|_{L^p}^{1-\f2p}\|\D
X(t)\|_{L^p}^{\f2p}+\|\D \p_Xv(t)\|_{L^p}\Bigr)\\
\leq &C\Bigl(\bigl(\|\na v(t)\|_{L^p}+\|\D v(t)\|_{L^p}\bigr)\|\D
X(t)\|_{L^p}\\
&\qquad\qquad\qquad\qquad+\|\D v(t)\|_{L^p}\|\na X(t)\|_{L^p}+\|\D
\p_Xv(t)\|_{L^p}\Bigr).
\end{split}
\eeno Applying Gronwall's inequality and using \eqref{S4eq2} leads
to \beq\label{S3eq5}\begin{split} \|\D X(t)\|_{L^p}\leq &C\Bigl(\|\D
X_0\|_{L^p}+ \|\D v\|_{L^1_t(L^p)}\|\na X\|_{L^\infty_t(L^p)}+\|\D
\p_Xv\|_{L^1_t(L^p)}\Bigr)\\
&\quad\times\exp\Bigl(C\bigl(\|\na
v\|_{L^1_t(L^p)}+\|\D v\|_{L^1_t(L^p)}\bigr)\Bigr)\\
\leq & C\cH_1(t)\bigl(1+\|\D \p_Xv\|_{L^1_t(L^p)}\bigr).
\end{split}
\eeq

 \no$\bullet$ \underline{The estimates of   $\Delta\d_X v$ and $\nabla\d_X\pi$}

We first get, by applying $\d_X$ to the momentum equation of
\eqref{inNS}, that \beq \label{S3eq8}
\begin{split}
&\d_t\d_X v + v\cdot\nabla \d_X v -  (1+a)\bigl( \Delta \d_X v -
\nabla \d_X \pi\bigr)
\\
&\qquad
 =\d_X a\bigl(\Delta
v-\nabla\pi\bigr) -(1+a)\Bigl([\Delta; \d_X]v -\sum_{i=1}^2 \nabla
X^i\p_i\pi\Bigr) \eqdefa F_1(v,\pi).
\end{split} \eeq
Taking space divergence to \eqref{S3eq8} gives \beno \div\bigl((1+a)
\nabla\d_X\pi\bigr) =\div\bigl(-\d_t(\d_X v) -v\cdot\nabla \d_X v
+\Delta\d_X v +a\Delta \d_X v + F_1\bigr). \eeno Due to $\dive
v=0=\dive X,$ and \beno [\D; \p_X]f=\D X\cdot\na
f+2\sum_{i=1}^2\p_iX\cdot\na \p_i f, \eeno
 one has
\begin{align*}
-\div\d_t(\d_X v) &=-\div \bigl( \d_t X\cdot\nabla  v +
X\cdot\nabla \d_t v \bigr)
 \\
 &=-\div \bigl(\d_t X\cdot\nabla v\bigr)
 -\div\bigl(\d_t v\cdot\nabla X\bigr),
\end{align*}
and
\begin{align*}
\div  \Delta \d_X v &=\div\bigl([\Delta;\d_X] v\bigr) +\div(
\d_X\Delta v)
\\
&=\div\bigl(\Delta X\cdot\nabla v\bigr)
+2\sum_{i=1}^2\div(\p_iX\cdot\nabla\p_i v)
 +\div\bigl(\Delta v\cdot\nabla X\bigr).
\end{align*}
Therefore, one achieves \beq \label{S3eq9}
\begin{split}
\div\bigl((1+a)\nabla\d_X\pi\bigr) =&-\div(\d_t  X\cdot\nabla v )
-\div\bigl(v\cdot\nabla\d_X v\bigr) +\div\Bigl( \d_Xa(\Delta
v-\nabla\pi)\\
& -(\d_t
v-\Delta v)\cdot\nabla X-2a\nabla X:\nabla^2 v\\
&  +(1+a)\nabla X\cdot\nabla\pi
  \Bigr )-\div\bigl(a\Delta X\cdot\nabla v \bigr) +\div(a\Delta\d_X
  v),
\end{split} \eeq
from which, and $\|a\|_{L^\infty}$ is sufficiently small, we infer
\begin{align*}
\|\nabla\d_X \pi(t)\|_{  L^{p}} &\lesssim \|\d_t X\cdot\nabla
v(t)\|_{L^p} +\|v\cdot\nabla\d_X v(t)\|_{  L^{p}}
\\
&\quad +\bigl\|\bigl( \d_t v, \Delta v, \nabla\pi\bigr)(t)
\bigr\|_{L^p} \bigl\|\bigl(\nabla X, \d_Xa \bigr)(t)
\bigr\|_{L^\infty} \,
  \\
  &\quad+\, \|\nabla v(t)\|_{L^\infty} \|\Delta
X(t)\|_{L^p}
 +\|a_0\|_{L^\infty} \|\Delta\d_X v(t)\|_{ L^{p}}  .
\end{align*}
For any $r_1\in \bigl]1, \f{2k}{k+\e}\bigr[,$  let us take $\s_1,
\s_2\in \bigl]\f{2p}{p+2},p\bigr[,$ $r_{01}, r_{02}, r_{03} \in
]1,\infty],$ $s_{01}, s_{02} \in ]2,\infty[,$ and $r_{0} \in ]1,2[$
so that
$$
\frac{1}{r_1}= \frac {1}{r_{01}}+\f1{\s_1}
+\frac{1}{s_{01}}=\f1{s_{02}}+\f1{r_{03}}+\f1{\s_2}=\f1{s_{02}}
+\frac {1}{r_{02}}+\frac{1}{r_{0}}.
$$
Then by virtue of \eqref{S4eq2} and the equation for $X$
\eqref{S1eq2}, we write
\begin{align*}
\|\d_t X&\cdot\nabla v\|_{L^{r_1}_t(L^p)} +\|v\cdot\nabla\d_X
v\|_{L^{r_1}_t(L^p)}
\\
\lesssim &
 \|\d_t X\otimes\nabla v\|_{L^{r_1}_t(L^p)}
+\|v\otimes\nabla X\otimes\nabla v\|_{ L^{r_1}_t(L^p)}
 +\|  v\otimes X\otimes \nabla^2 v\|_{ L^{r_1}_t(L^p)}
 \\
 \lesssim &
 \|X\|_{L^{r_{01}}_t(L^\infty)}
 \|\nabla v\|_{L^{\s_{1}}_t(L^\infty)}
 \|\nabla v\|_{L^{s_{01}}_t(L^p)}
 \\
 &
 +\|v\|_{L^{s_{02}}_t(L^\infty)}
 \Bigl( \|\nabla X\|_{L^{r_{03}}_t(L^p)}
 \|\nabla v\|_{L^{\s_{2}}_t(L^\infty)}
 +
 \|X\|_{L^{r_{02}}_t(L^\infty)}
 \|\nabla^2 v\|_{L^{r_{0}}_t(L^p)}
 \Bigr)
 \\
  \lesssim & \cH_1(t).
\end{align*}
Therefore, taking into account \eqref{S3eq2},  we obtain
\beq\label{S3eq10}
\begin{split}
\|\nabla\d_X \pi\|_{L^{r_1}_t (L^{p})} &\leq
\cH_1(t)+C\biggl(\bigl\|\bigl( \d_t v, \Delta v, \nabla\pi\bigr)
\bigr\|_{L^{r_1}_t(L^p)}
 \bigl\|\bigl( \nabla X, \d_{X_0}a_0
\bigr)\bigr\|_{L^\infty_t (L^\infty)} \\
& \quad+\Bigl(\int_0^t\|\na v(t')\|_{L^\infty}^{r_1}\|\D
X(t')\|_{L^p}^{r_1}\,dt'\Bigr)^{\f1{r_1}}
+\|a_0\|_{L^\infty}\|\Delta\d_X v\|_{L^{r_1}_t (L^{p})}\biggr)
\\
&\leq C\biggl(\cH_1(t)\bigl(1+\|\na
X\|_{L^\infty_t(L^\infty)}\bigr)+\|a_0\|_{L^\infty}\|\Delta\d_X
v\|_{L^{r_1}_t (L^{p})}\\
&\qquad\qquad\qquad\qquad\qquad+\Bigl(\int_0^t\|\na
v(t')\|_{L^\infty}^{r_1}\|\D
X(t')\|_{L^p}^{r_1}\,dt'\Bigr)^{\f1{r_1}} \biggr).
\end{split} \eeq
Similarly it follows that \beq\label{S3eq11}
\begin{split}
\|F_1\|_{L^{r_1}_t(L^p)}\leq &\|\p_X
a\|_{L^\infty_t(L^\infty)}\bigl(\|\D
v\|_{L^{r_1}_t(L^p)}+\|\na\pi\|_{L^{r_1}_t(L^p)}\bigr)\\
&+\|\na
X\|_{L^\infty_t(L^\infty)}\|\na^2v\|_{L^{r_1}_t(L^p)}+\Bigl(\int_0^t\|\na
v(t')\|_{L^\infty}^{r_1}\|\D
X(t')\|_{L^p}^{r_1}\,dt'\Bigr)^{\f1{r_1}}\\
\leq &\cH_1(t)\bigl(1+\|\na
X\|_{L^\infty_t(L^\infty)}\bigr)+\Bigl(\int_0^t\|\na
v(t')\|_{L^\infty}^{r_1}\|\D
X(t')\|_{L^p}^{r_1}\,dt'\Bigr)^{\f1{r_1}}.
\end{split}
\eeq

On the other hand, in view of \eqref{S3eq8}, we write \beq
\label{S3eq15}
\p_Xv(t)=e^{t\D}\p_{X_0}v_0+\int_0^te^{(t-t')\D}\bigl(-v\cdot\na\p_X
v+a\D\p_Xv-(1+a)\na\p_X\pi+F_1\bigr)(t')\,dt', \eeq from which,
\eqref{S3eq0} and maximal regularity estimate for heat semi-group,
we infer \beno \begin{split}
\|\p_t\p_Xv&\|_{L^{r_1}_t(L^p)}+\|\D\p_Xv\|_{L^{r_1}_t(L^p)}\\
\leq
&\|\p_{X_0}v_0\|_{W^{1-\f\e{k},p}}+C\Bigl(\|F_1\|_{L^{r_1}_t(L^p)}+\|v\cdot\na\p_X
v\|_{L^{r_1}_t(L^p)}\\
&+\|a\|_{L^\infty_t(L^\infty)}\|\D\p_Xv\|_{L^{r_1}_t(L^p)}+\bigl(1+\|a\|_{L^\infty_t(L^\infty)}\bigr)
\|\na\p_X\pi\|_{L^{r_1}_t(L^p)}\Bigr).
\end{split}
\eeno Hence by virtue of \eqref{S3eq10} and \eqref{S3eq11}, we get,
by taking $\|a_0\|_{L^\infty}$ to be sufficiently small, that \beq
\label{S3eq12}\begin{split} \|\p_t\p_Xv\|_{L^{r_1}_t(L^p)}+
 \|\D\p_Xv\|_{L^{r_1}_t(L^p)}\leq &\cH_1(t)\bigl(1+\|\na
X\|_{L^\infty_t(L^\infty)}\bigr)\\
&+C\Bigl(\int_0^t\|\na v(t')\|_{L^\infty}^{r_1}\|\D
X(t')\|_{L^p}^{r_1}\,dt'\Bigr)^{\f1{r_1}}. \end{split} \eeq Resuming
the above estimate into \eqref{S3eq5} and using \eqref{S3eq2a} gives
rise to \beno \begin{split} \|\D X\|_{L^\infty_t(L^p)}\leq &
\cH_1(t)\bigl(1+\|\na X\|_{L^\infty_t(L^p)}^{1-\f2p}\|\D
X\|_{L^\infty_t(L^p)}^{\f2p}\bigr)\\
&+Ct^{1-\f1{r_1}}\Bigl(\int_0^t\|\na v(t')\|_{L^\infty}^{r_1}\|\D
X(t')\|_{L^p}^{r_1}\,dt'\Bigr)^{\f1{r_1}}\\
\leq & \cH_1(t)+\f12\|\D
X\|_{L^\infty_t(L^p)}+Ct^{1-\f1{r_1}}\Bigl(\int_0^t\|\na
v(t')\|_{L^\infty}^{r_1}\|\D
X(t')\|_{L^p}^{r_1}\,dt'\Bigr)^{\f1{r_1}}, \end{split} \eeno which
gives \beno \|\D X\|_{L^\infty_t(L^p)}^{r_1}\leq
\cH_1(t)^{r_1}+Ct^{r_1-1}\int_0^t\|\na v(t')\|_{L^\infty}^{r_1}\|\D
X(t')\|_{L^p}^{r_1}\,dt'. \eeno Applying Gronwall's inequality and
using \eqref{S4eq2} gives \beq \label{S3eq13} \|\D
X\|_{L^\infty_t(L^p)}^{r_1}\leq
\cH(t)\exp\Bigl(Ct^{r_1-1}\int_0^t\|\na
v(t')\|_{L^\infty}^{r_1}\,dt'\Bigr)\leq \cH_1(t). \eeq  This
together with \eqref{S3eq12} and \eqref{S3eq10} ensures that
\beq\label{S3eq14} \|\p_t\p_X v\|_{L^{r_1}_t(L^p)}+\|\D
\p_Xv\|_{L^{r_1}_t(L^p)}+\|\na\p_X\pi\|_{L^{r_1}_t(L^p)}\leq
\cH_1(t). \eeq

 \no$\bullet$ \underline{The estimate of $\nabla\d_X v$ and $\p_tX$}

It is easy to observe that  for any $s_1\in
\big]2,\frac{2k}{\e}\bigr[$,  there holds $r_1=\f{2s_1}{2+s_1}\in
\bigl]1,\frac{2k}{k+\e}\bigr[.$ Hence in view of \eqref{S2eq14},
\eqref{S3eq0} and \eqref{S3eq15}, we get \beno
 \begin{split} \|\na\p_Xv\|_{L^{s_1}_t(L^p)}\leq
&\|\p_{X_0}v_0\|_{W^{1-\f\e{k},p}}+C\Bigl(\|F_1\|_{L^{r_1}_t(L^p)}+\|v\cdot\na\p_X
v\|_{L^{r_1}_t(L^p)}\\
&+\|a\|_{L^\infty_t(L^\infty)}\|\D\p_Xv\|_{L^{r_1}_t(L^p)}+\bigl(1+\|a\|_{L^\infty_t(L^\infty)}\bigr)
\|\na\p_X\pi\|_{L^{r_1}_t(L^p)}\Bigr).
\end{split}
\eeno which together with \eqref{S3eq14} implies that \beq
\label{S3eq16} \begin{split}  &\|\na\p_Xv\|_{L^{s_1}_t(L^p)}\leq
\cH_1(t),\andf\\
&\|\na\d_X v\|_{L^{r_1}_t(L^\infty)} \leq
C\|\na\p_Xv\|_{L^{r_1}_t(L^p)}^{1-\f2p}\|\D\p_Xv\|_{L^{r_1}_t(L^p)}^{\f2p}\leq
\cH_1(t). \end{split} \eeq Whereas since $\nabla v\in
L^{\f{2r}{2-r}}_t(L^p)$ for all $r\in ]1,2[$, one has $\p_Xv\in
L^{\f{2r}{2-r}}_t(L^p)$ and \beq\label{S3eq17} \|\d_X
v\|_{L^{s_1}_t(L^\infty)} \leq
C\|\p_Xv\|_{L^{s_1}_t(L^p)}^{1-\f2p}\|\na\p_Xv\|_{L^{s_1}_t(L^p)}^{\f2p}\leq
\cH_1(t). \eeq  Moreover,  it follows from \eqref{S1eq2} that
 \beno
\|\p_tX\|_{L^{s_1}_t(L^p)}\leq \|v\|_{L^{s_1}_t(L^\infty)}\|\na
X\|_{L^\infty_t(L^p)}+\|\p_Xv\|_{L^{s_1}_t(L^p)}\leq \cH_1(t), \eeno
and
\begin{align*}
\|\nabla\d_t X\|_{L^{s_1}_t(L^p)} &\lesssim \|\nabla v\otimes\nabla
X\|_{L^{s_1}_t(L^p)} +\|v\otimes\nabla^2 X\|_{L^{s_1}_t(L^p)}
+\|\nabla\d_X v\|_{L^{s_1}_t(L^p)}
\\
&\lesssim \|\nabla v\|_{L^{s_1}_t(L^p)} \|\nabla
X\|_{L^\infty_t(L^\infty)} + \|v\|_{L^{s_1}_t(L^\infty)} \|\nabla^2
X\|_{L^\infty_t(L^p)} +\|\nabla\d_X v\|_{L^{s_1}_t(L^p)}
\\
&\leq \cH_1(t).
\end{align*}
This yields \beno \|\d_t X\|_{L^{s_1}_t(W^{1,p})} \leq \cH_1(t).
\eeno This together with (\ref{S3eq13}-\ref{S3eq17}) leads to
\eqref{S3eq4a}.
\end{proof}

\setcounter{equation}{0}
\section{The proof of Theorem \ref{thm1}} \label{S6}

Recalling the definitions of   $J_0, J_1$ given by  \eqref{S4eq2}
and \eqref{S3eq4a} respectively, in general for $\ell\geq 2$ and any
$ r_\ell\in \bigl]1, \frac{2k}{k+\ell\,\e}\bigr[, \, s_\ell\in
\bigl] 2, \frac{2k}{\ell\,\e} \bigr[, $  let us define  $J_\ell(t)$
as follows \beq\label{S5eq1}
\begin{split}
&J_{\ell}(t)\eqdefa J_{\ell-1}(t)
 + \|(\d_t \d_X^{\ell } v,\Delta\d_X^{\ell} v, \nabla\d_X^{\ell
}\pi)\|_{L^{r_{\ell }}_t(L^p)} +\|\d_X^{\ell
}a\|_{L^\infty_t(L^\infty)}\\
&+\|\nabla\d_X^{\ell } v\|_{L^{r_{\ell }}_t(L^\infty) \cap
L^{s_{\ell }}_t(L^p)}+\|\d_X^{\ell } v\|_{L^{s_{\ell }}_t(L^\infty)}
+ \|\d_t \d_X^{\ell-1}X\|_{L^{s_{\ell }}_t(W^{1,p})} +
\|\d_X^{\ell-1} X\|_{L^\infty_t(W^{2,p})}.
\end{split} \eeq
Then Theorem \ref{thm1} essentially follows from the following
proposition:

\begin{prop}\label{S5prop1}
{\sl Under the assumptions of Proposition \ref{S3prop1}, one has
\beq \label{S5eq2} J_\ell(t)\leq \cH_\ell(t), \quad \forall\ \ell
=2,\cdots, k. \eeq }
\end{prop}

 Its
proof will depend on the following three lemmas:

\begin{lem}\label{S5lem1}
{\sl For any $\ell \in\{1,\cdots,k\, \},$ $r_\ell \in \bigl]1,
\frac{2k}{k+\ell\,\e}\bigr[,$ and $ s_\ell\in
\bigl]2,\frac{2k}{\ell\,\e}\bigr[,$  we have \beq
\label{S5eq3}\begin{split} &\bigl\|\d_X^i\nabla\d_X^{\ell-i}
X-\nabla\d_X^\ell X\bigr\|_{L^\infty_t(L^\infty)}+
\bigl\|\d_X^i\nabla \d_X^{\ell-i} X-\nabla \d_X^\ell
X\|_{L^\infty_t(W^{1,p})}\\
&+ \bigl\|\d_X^i\nabla^2\d_X^{\ell-i} X-\nabla^2 \d_X^\ell
X\bigr\|_{L^\infty_t(L^p)}+\bigl\|\d_X^i\d_t\d_X^{\ell-i}
X-\d_t\d_X^\ell X\bigr\|_{L^{s_\ell}_t(W^{1,p})} \lesssim
J_{\ell}^{\ell+1}.
\end{split} \eeq
When $i\neq \ell$, one has \beq \label{S5eq4}\begin{split}
&\bigl\|\d_X^i \nabla \d_X^{\ell-i} v-\nabla\d_X^\ell v\bigr\|
_{L^{r_{\ell-1}}_t(L^\infty)\cap
L^{s_{\ell-1}}_t(L^p)}+\bigl\|\d_X^i \nabla^2 \d_X^{\ell-i}
v-\nabla^2\d_X^\ell v\bigr\| _{L^{r_{\ell-1}}_t(L^p)}
\\
&+\bigl\|\d_X^i \nabla \d_X^{\ell-i} \pi-\nabla\d_X^\ell
\pi\bigr\|_{L^{r_{\ell-1}}_t(L^p) }+\bigl\|\d_X^i \d_t\d_X^{\ell-i}
v-\d_t\d_X^\ell v\| _{L^{r_{\ell-1}}_t(L^p) } \lesssim
J_{\ell-1}^{\ell+1} ,
\end{split} \eeq
and in the case when $i=\ell$, there holds \beq
\label{S5eq5}\begin{split} &\bigl\|\d_X^\ell\nabla v-\nabla\d_X^\ell
v +\d_X^{\ell-1}\nabla X\cdot\nabla v
\bigr\|_{L^{r_{\ell-1}}_t(L^\infty)\cap L^{s_{\ell-1}}_t(L^p)}
\lesssim J_{\ell-1}^{\ell+1}
\\
&\bigl\|\d_X^\ell \Delta   v - \Delta\d_X^\ell v +\d_X^{\ell-1}
\Delta X\cdot\nabla v +2\d_X^{\ell-1}\nabla X:\nabla^2
v\bigr\|_{L^{r_{\ell-1}}_t(L^p)} \lesssim J_{\ell-1}^{\ell+1},
\\
&\bigl\|\d_X^\ell \nabla   \pi-\nabla\d_X^\ell \pi
+\d_X^{\ell-1}\nabla X\cdot\nabla\pi\bigr\| _{L^{r_{\ell-1}}_t(L^p)
} \lesssim J_{\ell-1}^{\ell+1},
\\
&\bigl\|\d_X^\ell \d_t v-\d_t\d_X^\ell v +\d_X^{\ell-1}\d_t
X\cdot\nabla v\bigr\| _{L^{r_{\ell-1}}_t(L^p) } \lesssim
J_{\ell-1}^{\ell+1}.
\end{split} \eeq}
\end{lem}

\begin{rmk}\label{S5rmk1}
It follows  from \eqref{S5eq3}, \eqref{S5eq4} and \eqref{S5eq5} that
\beq\label{S5eq5a}
\begin{split}
& \|\d_X^i\nabla\d_X^{\ell-i} X \|_{L^\infty_t(L^\infty)}+
\|\d_X^i\nabla \d_X^{\ell-i} X \|_{L^\infty_t(W^{1,p})},
\\
&\qquad\qquad+\|\d_X^i\nabla^2\d_X^{\ell-i} X \|_{L^\infty_t(L^p)}+
\|\d_X^i\d_t\d_X^{\ell-i} X \|_{L^{s_{\ell+1}}_t(W^{1,p})} \lesssim
J_{\ell+1}^{\ell+1},
\end{split} \eeq
and \beq\label{S5eq5b}
\begin{split}
&\|\d_X^i \nabla \d_X^{\ell-i} v \| _{L^{r_{\ell }}_t(L^\infty)\cap
L^{s_{\ell}}_t(L^p)} +\|\d_X^i \nabla^2 \d_X^{\ell-i} v
\|_{L^{r_{\ell }}_t(L^p)}
\\
&\qquad\qquad\qquad+\|\d_X^i \nabla \d_X^{\ell-i} \pi \|
_{L^{r_{\ell }}_t(L^p) }+\|\d_X^i \d_t\d_X^{\ell-i} v \|
_{L^{r_{\ell }}_t(L^p) } \lesssim J_{\ell}^{\ell+1}.
\end{split} \eeq
\end{rmk}

\begin{proof} It is easy to observe that
\beno
\begin{split}
&\|\p_X\na X-\na\p_XX\|_{L^\infty_t(L^\infty)}+\|\p_X\na
X-\na\p_XX\|_{L^\infty_t(W^{1,p})}\\
&+\|\p_X\na^2X-\na^2\p_XX\|_{L^\infty_t(L^p)}+\|\p_X\p_tX-\p_t\p_XX\|_{L^{s_1}_t(W^{1,p})}\\
&\leq C\|\na X\|_{L^\infty_t(W^{1,p})}\bigl(\|\na
X\|_{L^\infty_t(W^{1,p})}+\|\p_tX\|_{L^{s_1}_t(W^{1,p})}\bigr)\lesssim
J_1^2.
\end{split}
\eeno This shows that \eqref{S5eq3} holds for $\ell=1.$ And
\eqref{S5eq4} and \eqref{S5eq5} hold trivially for $\ell=1.$ And
hence Lemma \ref{S5lem1} holds for $k=1.$

Let us now assume that (\ref{S5eq3}-\ref{S5eq5b}) hold for
$\ell\leq j-1$ for $j\leq k.$ We are going to prove that
(\ref{S5eq3}-\ref{S5eq5})  also hold for $\ell=j,$  which also
implies \eqref{S5eq5a} and \eqref{S5eq5b} for $\ell=j.$ As a matter
of fact, it follows from a trivial calculation that for $i\leq j-1,$
\begin{align*}
\d_X^{i+1}\nabla\d_X^{j-i-1} f &=\nabla\d_X^j f+\sum_{m=0}^{i}
\d_X^m[\d_X,\nabla]\d_X^{j-m-1}f
\\
&=\nabla\d_X^j f -\sum_{m=0}^{i} \d_X^m(\nabla X\cdot\nabla
\d_X^{j-m-1}f)
\\
&=\nabla\d_X^j f -\sum_{m=0}^{i}\sum_{n=0}^{m} C_m^n \d_X^n\nabla X
\cdot\d_X^{m-n}\nabla \d_X^{j-m-1}f,
\end{align*}
from which and  the induction assumptions, we infer
\beq\label{S5eq6}
\begin{split}
&\bigl\|\d_X^{i+1}\nabla\d_X^{j-i-1} X-\nabla\d_X^j
X\bigr\|_{L^\infty_t(L^\infty)}+\bigl\|\d_X^{i+1}\nabla\d_X^{j-i-1}
X-\nabla\d_X^j X\bigr\|_{L^\infty_t(W^{1,p})}\\
&\lesssim \sum_{m=0}^{i}\sum_{n=0}^{m} \|\p_X^n\na
X\|_{L^\infty_t(W^{1,p})}\|\p_X^{m-n}\na\p_X^{j-m-1}X\|_{L^\infty_t(W^{1,p})}\\
&\lesssim \sum_{m=0}^{i}\sum_{n=0}^{m}
 J_{n+1}^{n+1}  J_{j-n }^{j-n}
\lesssim J_j^{j+1}.
\end{split} \eeq

By the same manner, it also easy to observe that for $i\leq j-1,$
\begin{align*}
\d_X^{i+1}\nabla^2\d_X^{j-i-1} f &=\nabla^2\d_X^j f +\sum_{m=0}^{i}
\d_X^m[\d_X,\nabla^2]\d_X^{j-m-1}f
\\
&=\nabla^2\d_X^j f -\sum_{m=0}^{i}\d_X^m \bigl( \nabla^2
X\cdot\nabla \d_X^{j-m-1}f + 2\nabla X\cdot\nabla^2 \d_X^{j-m-1}f
\bigr)
\\
&=\nabla^2\d_X^j f -\sum_{m=0}^{i}\sum_{n=0}^{m} C_m^n\Bigl(
\d_X^n\nabla^2 X \cdot\d_X^{m-n}\nabla \d_X^{j-m-1}f\\
&\qquad\qquad\qquad\qquad\qquad\quad+2 \d_X^n(\nabla X)
\cdot\d_X^{m-n}\nabla^2 \d_X^{j-m-1}f \Bigr).
\end{align*}
This together with the inductive assumptions ensures that \beq
\label{S5eq7}
\begin{split}
\|\p_X^{i+1}&\na^2\p_X^{j-i-1}X-\na^2\p_X^jX\|_{L^\infty_t(L^p)}\\
&\lesssim
\sum_{m=0}^i\sum_{n=0}^m\Bigl(\|\p_X^{m-n}\na\p_X^{j-m-1}X\|_{L^\infty_t(L^\infty)}\|\p_X^n\na^2X\|_{L^\infty_t(L^p)}\\
&\qquad\qquad\quad
+\|\p_X^{m-n}\na^2\p_X^{j-m-1}X\|_{L^\infty_t(L^p)}\|\p_X^n\na
X\|_{L^\infty_t(L^\infty)}\Bigr) \lesssim J_j^{j+1}.
\end{split}
\eeq

Finally notice that
\begin{align*}
\d_X^{i+1}\d_t\d_X^{j-i-1} f&=\d_t\d_X^j f +\sum_{m=0}^{i}
\d_X^m[\d_X,\d_t]\d_X^{j-m-1}f
\\
&=\d_t\d_X^j f -\sum_{m=0}^{i} \d_X^m( \d_t X\cdot\nabla
\d_X^{j-m-1}f )
\\
&=\d_t\d_X^j f -\sum_{m=0}^{i}\sum_{n=0}^{m} C_m^n  \d_X^n\d_t X
\cdot\d_X^{m-n}\nabla \d_X^{j-m-1}f,
\end{align*}
which together with the inductive assumptions ensures that
\beq\label{S5eq8}
\begin{split}
&\bigl\|\p_X^{i+1}\p_t\p_X^{j-i-1}X-\p_t\p_X^jX\bigr\|_{L^{s_j}_t(W^{1,p})}\\
&\qquad\lesssim \sum_{m=0}^{i}\sum_{n=0}^{m} C_m^n\|\d_X^n\d_t
X\|_{L^{s_j}_t(W^{1,p})} \|\d_X^{m-n}\nabla
\d_X^{j-m-1}X\|_{L^\infty_t(W^{1,p})}\lesssim J_j^{j+1}.
\end{split}
\eeq \eqref{S5eq6} together with \eqref{S5eq7} and \eqref{S5eq8}
shows that \eqref{S5eq3} holds for $\ell=j.$

Exactly by the same manner, and using the fact that for any $r_j\in
\bigl]1,\frac{2k}{k+j\e}\bigr[$, there exist $s_{n+1}\in \bigl]2,
\f{2k}{(n+1)\e}\bigr[$ and $ s_{j-n-1}\in \bigl]2,
\f{2k}{(j-n-1)\e}\bigr[$, $n=0,\cdots, j-1$ such that
$$
\frac{1}{r_j} = \frac{1}{s_{n+1}} +\frac{1}{s_{j-n-1}},
$$
we can also prove that \eqref{S5eq4} and \eqref{S5eq5} holds for
$\ell=j.$ This completes the proof of Lemma \ref{S5lem1}.
\end{proof}

Before proceeding, let us  take the operator $\d_X^{\ell-1}$ to
\eqref{S3eq8} to get \beq\label{S5eq9} \d_t\d_X^\ell v
+v\cdot\nabla\d_X^\ell v -(1+a)\bigl(\Delta\d_X^\ell
v-\nabla\d_X^\ell\pi\bigr) =F_\ell(v,\pi), \quad\forall
\ell=2,\cdots,k. \eeq where $F_\ell(v,\pi)$ can be inductively
defined as \beq\label{S5eq10} F_\ell(v,\pi)
=F_1(\d_X^{\ell-1}v,\d_X^{\ell-1}\pi) +\d_X F_{\ell-1}(v,\pi),
\quad\ell\geq 2. \eeq

\begin{lem}\label{S5lem2}
{\sl For $\ell=2,\cdots,k$ and   for all $ r_{\ell}\in
\bigl]1,\frac{2k}{k+\ell\e}\bigr[,$  there holds \beq\label{S5eq12}
\begin{split}
\|  F_{\ell }\|_{L^{r_{\ell}}_t(L^p)} \lesssim &
  J_{\ell-1}^{\ell+2}
+J_0\|\d_{X_0}^{\ell}a_0\|_{L^\infty_t(L^\infty)}+\Bigl(\int_0^t\Bigl(\|\na
v(t')\|_{L^\infty}^{r_\ell}\|\Delta\p_X^{\ell-1}
X(t')\|_{L^p}^{r_\ell}\\
&\qquad\qquad\quad+\bigl(\|\na^2
v(t')\|_{L^p}^{r_\ell}+\|\na\pi(t')\|_{L^p}^{r_\ell}\bigr)
\|\nabla \d_X^{\ell-1}X(t')\|_{L^\infty}^{r_\ell}\Bigr)\,dt'\Bigr)^{\f1{r_\ell}}.
\end{split} \eeq}
\end{lem}

\begin{proof}
 In view of \eqref{S5eq10}, one gets, by induction, that
$$
F_\ell(v,\pi)
 =F_1(\d_X^{\ell-1}v,\d_X^{\ell-1}\pi)
+\d_X F_{\ell-1}(v,\pi) =\sum_{i=0}^{\ell-1}
\d_X^{\ell-1-i}F_1(\d_X^i v,\d_X^i \pi),
$$
which along with the definition of $F_1$ given by \eqref{S3eq8}, we
write
\begin{align*}
F_\ell(v,\pi) =\sum_{i=0}^{\ell-1} \d_X^{\ell-1-i} \Bigl(&
\d_Xa\bigl(\Delta \d_X^i v-\nabla\d_X^i \pi\bigr)
\\
&-(1+a)\bigl(2\nabla X:\nabla^2\d_X^i v - \nabla X\cdot\nabla\d_X^i
\pi + \Delta X\cdot\nabla \d_X^i v\bigr)\Bigr).
\end{align*}
Correspondingly $F_\ell$ reads
\begin{align*}
F_\ell(v,\pi) =&\sum_{i=0}^{\ell-1}\sum_{n=0}^{\ell-1-i}
C^n_{\ell-1-i}\Bigl( \d_X^{n+1}a\, \d_X^{\ell-1-i-n} \bigl(\Delta
\d_X^i v-\nabla\d_X^i \pi\bigr)
\\
&\qquad\qquad\qquad\quad -2 \d_X^{n }(1+a)\, \sum_{m=0}^{\ell-1-i-n}
C_{\ell-1-i-n}^m \d_X^m\nabla X :\d_X^{\ell-1-i-n-m}\nabla^2\d_X^i v
\\
&\qquad\qquad\qquad\quad +\d_X^{n}(1+a) \sum_{m=0}^{\ell-1-i-n}
C_{\ell-1-i-n}^m \d_X^m\nabla X \cdot\d_X^{\ell-1-i-n-m}\nabla\d_X^i
\pi
\\
&\qquad\qquad\qquad\quad - \d_X^{n }(1+a)\, \sum_{m=0}^{\ell-1-i-n}
C_{\ell-1-i-n}^m \d_X^m\Delta X \cdot\d_X^{\ell-1-i-n-m}\nabla\d_X^i
v\Bigr)
\\
\eqdefa & F_\ell^1+F_\ell^2 +F_\ell^3+F_\ell^4.
\end{align*}
It is easy to observe that
\begin{align*}
F_\ell^1 &=\sum_{i=1}^{\ell-1}\sum_{n=0}^{\ell-1-i} C^j_{\ell-1-i}
\d_X^{n+1}a\, \d_X^{\ell-1-i-n} (\Delta \d_X^i v-\nabla\d_X^i \pi)
\\
&\quad +\sum_{n=0}^{\ell-2} C^n_{\ell-1} \d_X^{n+1}a\,
\d_X^{\ell-1-n} (\Delta  v-\nabla \pi) +\d_X^{\ell}a\, (\Delta
v-\nabla \pi),
\end{align*}
from which and Lemma \ref{S5lem1}, we infer \beno \begin{split}
\bigl\|F_\ell^1-& \d_X^{\ell}a\bigl(\Delta   v-\nabla
\pi\bigr)\bigr\| _{L^{r_{\ell-1}}_t(L^p)}\\
 \leq &
\sum_{i=1}^{\ell-1}\sum_{n=0}^{\ell-1-i} C^n_{\ell-1-i}
\|\d_X^{n+1}a\|_{L^\infty_t(L^\infty)}\|\d_X^{\ell-1-i-n} (\Delta
\d_X^i v-\nabla\d_X^i \pi)\| _{L^{r_{\ell-1}}_t(L^p)}
\\
&\qquad\qquad +\sum_{n=0}^{\ell-2} C^n_{\ell-1}
\|\d_X^{n+1}a\|_{L^\infty_t(L^\infty)}\| \d_X^{\ell-1-n}(\Delta
v-\nabla \pi)\| _{L^{r_{\ell-1}}_t(L^p)} \lesssim
J_{\ell-1}^{\ell+2}.
\end{split}
\eeno By the same manner, we also have
$$\longformule{\bigl\|
 F_\ell^2+
  2(1+a)
 \d_X^{\ell-1} \nabla X
 :\nabla^2 v\bigr\| _{L^{r_{\ell-1}}_t(L^p)}+\bigl\|
 F_\ell^3-
  (1+a)
 \d_X^{\ell-1} \nabla X
 \cdot\nabla \pi\bigr\| _{L^{r_{\ell-1}}_t(L^p)}}{{}+\bigl\|
  F_\ell^4+
  (1+a)
 \d_X^{\ell-1} \Delta X
 \cdot\nabla v\bigr\| _{L^{r_{\ell-1}}_t(L^p)}
\lesssim J_{\ell-1}^{\ell+2}.} $$
As a consequence, we conclude
\beq\label{S5eq11}
\begin{split}
&\Bigl\| F_\ell (v,\pi) - \d_X^{\ell}a\, \bigl(\Delta
v-\nabla \pi\bigr)
 +(1+a)\bigl(2
 \d_X^{\ell-1} \nabla X
 :\nabla^2 v
 \\
 &\qquad\qquad\qquad\qquad\quad-
 \d_X^{\ell-1} \nabla X
 \cdot\nabla \pi
+
 \d_X^{\ell-1} \Delta X
 \cdot\nabla v\bigr)
 \Bigr\|_{L^{r_{\ell-1}}_t(L^p)}
\lesssim J_{\ell-1}^{\ell+2},
\end{split} \eeq
which ensures \eqref{S5eq12} thanks to \eqref{S4eq2}, \eqref{S3eq2}
and \eqref{S5eq3}. This completes the proof of Lemma \ref{S5lem2}.
\end{proof}

We also need the following estimate for the pressure function.

\begin{lem}\label{S5lem3}
{\sl For $\ell=2,\cdots, k$ and $r_\ell\in \bigl]1,
\f{2k}{k+\ell\e}\bigr[,$ one has \beq\label{S5eq13}
\begin{split}
\| \na\p_X^\ell\pi\|_{L^{r_{\ell}}_t(L^p)}& \lesssim
  J_{\ell-1}^{\ell+2}
+J_0\|\d_{X_0}^{\ell}a_0\|_{L^\infty_t(L^\infty)}+\|a_0\|_{L^\infty}\|\D\p_X^\ell
v\|_{L^{r_\ell}_t(L^p)}\\
&\quad+\Bigl(\int_0^t
\|a_0\|_{L^\infty} \|\nabla v(t')\|_{L^\infty}^{r_\ell}\|\Delta \d_X^{\ell-1}X(t')\|_{L^p}\\
&\qquad +\|(v\cdot\nabla v(t'), \Delta v(t'),
\na\pi(t'))\|_{L^p}^{r_\ell}\|\na\d_X^{\ell-1}X(t')\|_{L^\infty}^{r_\ell}\,dt'\Bigr)^{\f1{r_\ell}}.
\end{split} \eeq}
\end{lem}

\begin{proof} Motivated by \eqref{S2eq11a} and \eqref{S3eq9},  we assume that $\p_X^{\ell-1}\pi$ verifies
\beno \dive\bigl((1+a)\na \p_X^{\ell-1}\pi\bigr)=\dive G_{\ell-1}.
\eeno As $\dive X=0,$ we have \beno \p_X\dive
f=\dive(\p_Xf)-\dive(f\cdot\na X), \eeno so that we have
\beq\label{S5eq13a}
\begin{split}
\dive\bigl((1+a)\na
\p_X^{\ell}\pi\bigr)=\dive\Bigl(&\p_XG_{\ell-1}-G_{\ell-1}\cdot\na
X-\p_Xa\na\p_X^{\ell-1}\pi\\
&+(1+a)\bigl(\na\p_X^{\ell-1}\pi\cdot\na X+\na
X\cdot\na\p_X^{\ell-1}\pi\bigr)\Bigr)\eqdefa \dive G_\ell.
\end{split}
\eeq

Let us assume by induction that for $0\leq i\leq j\leq \ell-1$,
\beq\label{S5eq13b} \|\p_X^iG_{\ell-1-j}\|_{L^{r_{\ell-1}}_t(L^p)}
\lesssim J_{\ell-1}^{\ell+1+i-j}. \eeq It is easy to observe from
\eqref{S5eq13a} that \beno
\begin{split}
\p_X^iG_{\ell-j}=\p_X^i\Bigl(&\p_XG_{\ell-1-j}
-G_{\ell-1-j}\cdot\na
X-\p_Xa\na\p_X^{\ell-1-j}\pi\\
&+(1+a)\bigl(\na\p_X^{\ell-1-j}\pi\cdot\na X+\na
X\cdot\na\p_X^{\ell-1-j}\pi\bigr)\Bigr).
\end{split}
\eeno We shall decompose the estimate of $\p_X^iG_{\ell-j}$ with
$0\leq i\leq j\leq \ell$ into the following cases:
\begin{itemize}
\item[(1)]  $0\leq i< j\leq \ell-2.$
It follows from the inductive assumption \eqref{S5eq13b} and Remark
\ref{S5rmk1} that for $0\leq i\leq j\leq \ell-2$, \beno
\begin{split}
&\|\d_X^{i+1}G_{\ell-1-j}\|_{L^{r_{\ell-1}}_t(L^p)} \lesssim
J_{\ell-1}^{\ell+2+i-j}, \hbox{ if\  moreover }\ i<j,
\\
&\|\p_X^i(G_{\ell-1-j}\cdot\nabla X)\|_{L^{r_{\ell-1}}_t(L^p)}
\lesssim \sum_{0\leq m\leq i}
\|\p_X^m(G_{\ell-1-j})\|_{L^{r_{\ell-1}}_t(L^p)} \|\p_X^{i-m}\nabla
X\|_{L^\infty_t(L^\infty)}
\\
&\qquad\qquad\qquad\qquad\qquad\qquad\lesssim
J_{\ell-1}^{\ell+1+m-j} J_{i-m+1}^{i-m+1} \lesssim
J_{\ell-1}^{\ell+2+i-j},
\\
&\|\p_X^i(\p_X a\nabla\p_X^{\ell-1-j}\pi)\|_{L^{r_{\ell-1}}_t(L^p)}
\lesssim\sum_{0\leq m\leq i} \|\p_X^{m+1}a\|_{L^\infty_t(L^\infty)}
\|\p_X^{i-m}\nabla \p_X^{\ell-1-j}\pi\|_{L^{r_{\ell-1}}_t(L^p)}
\\
&\qquad\qquad\qquad\qquad\qquad\qquad\quad\lesssim J_{m+1}
J_{\ell-1+i-m-j}^{\ell+i-m-j} \lesssim J_{\ell-1}^{\ell+1+i-j},
\end{split}
\eeno and
\begin{align*}
&\Bigl\|\p_X^i\Bigl((1+a)\nabla\p_X^{\ell-1-j}\pi\otimes\nabla X\Bigr)\Bigr\|_{L^{r_{\ell-1}}_t(L^p)}
\\
&\lesssim\sum_{\substack{0\leq m\leq i\\
 0\leq n\leq i-m}} \|\p_X^m
(1+a)\|_{L^\infty_t(L^\infty)} \|\p_X^{n}\nabla
\p_X^{\ell-1-j}\pi\|_{L^{r_{\ell-1}}_t(L^p)} \|\p_X^{i-m-n}\nabla
X\|_{L^\infty_t(L^\infty)}
\\
&\lesssim
J_{m} J_{\ell-1+n-j}^{\ell+n-j} J_{i+1-m-n}^{i+1-m-n}
\lesssim
J_{\ell-1}^{\ell+2+i-j}.
\end{align*}
Hence we obtain
$$
\|\p_X^i G_{\ell-j}\|_{L^{r_{\ell-1}}_t(L^p)} \lesssim
J_{\ell-1}^{\ell+2+i-j}\ \  \hbox{ when }\ \ 0\leq i<j\leq \ell-2.
$$

\item[(2)]  $0\leq i=j\leq \ell-2.$  We first deduce from case (1) that
 \beno
\p_X^iG_{\ell-i}=\p_X^{i+1}G_{\ell-1-i}
+R_{\ell}^{\ell-i}\with
\|R_{\ell}^{\ell-i}\|_{L^{r_{\ell-1}}_t(L^p)}
\lesssim
J_{\ell-1}^{\ell+2}. \eeno
Along this line, we can show that
 \beq\label{S5eq13c}
 \begin{split}
&\p_X^iG_{\ell-i}=\p_X^{\ell-1}G_{1}+R_{\ell}^{2}\with
\|R_{\ell}^{2}\|_{L^{r_{\ell-1}}_t(L^p)}\lesssim
J_{\ell-1}^{\ell+2}.
\end{split}\eeq
This   shows that
\beno \|\p_X^iG_{\ell-i}-\p_X^{\ell-1} G_1\|_{L^{r_{\ell-1}}_t(L^p)}\lesssim
J_{\ell-1}^{\ell+2}\quad\mbox{for}\ \ i=0,\cdots \ell-2,\eeno
with
\begin{align*}
&G_1=\p_X G_0-G_0\cdot\nabla X-\p_X a\nabla\pi
+(1+a)(\nabla\pi\cdot\nabla X+\nabla X\cdot\nabla\pi),
\\
&G_0=-v\cdot\nabla v+a\Delta v.
\end{align*}
Noticing that
\begin{equation}\label{S5eq13d}
\begin{array}{l}
\Bigl\|\Bigl( \p_X^{\ell-1}G_1-\p_X^\ell G_0
+G_0\cdot \p_X^{\ell-1}\nabla X
+\p_X^\ell a\nabla\pi
\\
\quad
-(1+a)(\nabla\pi\cdot\p_X^{\ell-1} \nabla X
+\p_X^{\ell-1}\nabla X\cdot\nabla\pi)\Bigr)\Bigr\|_{L^{r_{\ell-1}}_t(L^p)}
\lesssim J_{\ell-1}^{\ell+2},
\\
\|\p_X^{\ell}G_0+\p_X^\ell v\cdot\nabla v+v\cdot\p_X^\ell \nabla v
-\p_X^\ell a\Delta v-a\p_X^\ell\Delta v\|_{L^{r_{\ell-1}}_t(L^p)}
\lesssim J_{\ell-1}^{\ell+2},
\end{array}\end{equation}
and
\beno
\begin{split}
\|v\cdot\na\p_X^\ell v\|_{L^{r_\ell}_t(L^p)}\lesssim
\|v\|_{L^s_t(L^\infty)}\bigl(&\|\na
X\|_{L^\infty_t(L^p)}\|\na\p_X^{\ell-1}v\|_{L^{r_{\ell-1}}_t(L^\infty)}\\
&+\|X\|_{L^\infty_t(L^\infty)}\|\D\p_X^{\ell-1}v\|_{L^{r_{\ell-1}}_t(L^p)}\bigr),
\end{split}
\eeno \beno \|\p_X^\ell v\cdot\na v\|_{L^{r_\ell}_t(L^p)}\leq
\|X\|_{L^\infty_t(L^\infty)}\|\na\p_X^{\ell-1}v\|_{L^{r_{\ell-1}}_t(L^\infty)}\|\na
v\|_{L^s_t(L^p)}, \eeno where $s\in ]2,\infty[$ so that
$\f1{r_\ell}=\f1s+\f1{r_{\ell-1}},$
we know that
$$\|\p_X^i G_{\ell-i}\|_{L^{r_\ell}_t(L^p)}\lesssim J_{\ell}^{\ell+2}.$$

\item[(3)]  $j=\ell-1$, $0\leq i\leq \ell-2.$ Note that when $j=\ell-1$,
$$
G_{\ell-1-j}=G_0=-v\cdot\nabla v+a\Delta v,
\quad \nabla\p_X^{\ell-1-j}\pi=\nabla\pi
$$
by the same arguments as in the case $0\leq i<j\leq \ell-2$ one obtains also the   estimate  $\|\p_X^i G_{\ell-j}\|_{L^{r_{\ell-1}}_t(L^p)}\lesssim J_{\ell-1}^{\ell+2+i-j}$ for  the case $0\leq i\leq \ell-2$, $j=\ell-1$.

\item[(4)]  $j=\ell$, $0\leq i\leq \ell-1.$ In this case, there  holds
$$
\|\p_X^i G_0\|_{L^{r_{\ell-1}}_t(L^p)}
\lesssim J_{\ell-1}^{i+2}.
$$

\item[(5)]  $i=j=\ell-1$ or $i=j=\ell.$  As in the case  when $0\leq i=j\leq
\ell-2$, one has
\begin{align*}
&\|\p_X^{\ell-1}G_1\|_{L^{r_{\ell}}_t(L^p)},
\,\|\p_X^{\ell}G_0\|_{L^{r_{\ell}}_t(L^p)} \lesssim
J_{\ell}^{\ell+2}.
\end{align*}
\end{itemize}
To conclude, we have arrived at the following for $0\leq i\leq j\leq \ell\leq k$:
$$
\|\p_X^i G_{\ell-j}\|_{L^{r_\ell}_t(L^p)}
\lesssim J_{\ell}^{\ell+2+i-j},
$$
and hence   \eqref{S5eq13b} follows by induction.

Thanks to \eqref{S5eq13b}, \eqref{S5eq13c} and \eqref{S5eq13d}, we infer
\beq\label{S5eq13e}
\begin{split}
\bigl\|G_\ell-\d_X^\ell a\Delta v-a\p_X^\ell\Delta v
+G_0\cdot\p_X^{\ell-1}\na X+\p_X ^\ell a \nabla\pi
\\
-(1+a)\bigl(\na\pi\cdot\p_X^{\ell-1}\na X+
&\p_X^{\ell-1}\na X\cdot\na\pi\bigr)\bigr\|_{L^{r_{\ell-1}}_t(L^p)}
\lesssim J_{\ell-1}^{\ell+2}.
\end{split} \eeq We apply \eqref{S5eq5} to deal with the term $[\d_X^\ell\Delta v; \Delta\d_X^\ell v]$
and \eqref{S5eq3} to the term $[\d_X^{\ell-1}\nabla X;
\nabla\d_X^{\ell-1}X]$. Then \eqref{S5eq13} follows from
\eqref{S5eq13a} and \eqref{S5eq13e}.

 \end{proof}

We  now turn to the proof of Proposition \ref{S5prop1}.

\begin{proof}[Proof of Proposition \ref{S5prop1}]

\no$\bullet$ \underline{The estimates of   $\p_t\d_X^\ell v,
\Delta\d_X^\ell v$ and $\na\p_X^\ell \pi$}

As in the previous sections, we reformulate \eqref{S5eq9} as
\beq\label{S5eq15}
\begin{split}
\d_X^\ell v(t) =e^{t\Delta}\d_{X_0}^\ell v_0 +\int^t_0
e^{(t-t')\Delta}\tilde F_\ell(t') \,dt'\with\\
\tilde F_\ell \eqdefa -v\cdot\nabla \d_X^\ell v +a\Delta \d_X^\ell v
-(1+a)\nabla\d_X^\ell \pi +F_\ell(v,\pi),
\end{split} \eeq
from which and \eqref{S3eq0}, we infer \beq\label{S5eq15a}
\|\p_t\p_X^\ell v\|_{L^{r_\ell}_t(L^p)}+\|\D\p_X^\ell
v\|_{L^{r_\ell}_t(L^p)}\leq
 C\bigl(\|\p_{X_0}^\ell v_0\|_{W^{1-\f{\ell}k\e,p}}+\|\tilde{F}_\ell\|_{L^{r_\ell}_t(L^p)}\bigr).
 \eeq
 Note that for any $r_\ell\in \bigl]1,\f{2k}{k+\ell\e}\bigr[,$ there
 exist $\s_1,\s_2\in ]2,\infty[$ and $r_{\ell-1}\in
 \bigl]1,\f{2k}{k+(\ell-1)\e}\bigr[,$ $s_{\ell-1}\in
 \bigl]2,\f{2k}{(\ell-1)\e}\bigr[$ so that
 $\f1{r_\ell}=\f1{\s_1}+\f1{s_{\ell-1}}=\f1{\s_2}+\f1{r_{\ell-1}}.$
  Then it follows from Remark
\ref{S5rmk1} that
  \beno
\begin{split}
\|v\cdot\na\p_X^\ell v\|_{L^{r_\ell}_t(L^p)}\lesssim &
\|v\|_{L^{\sigma_1}_t(L^\infty)}\|\na
X\|_{L^\infty_t(L^\infty)}\|\na\p_X^{\ell-1}
v\|_{L^{s_{\ell-1}}_t(L^p)}\\
&+\|v\|_{L^{\sigma_2}_t(L^\infty)}\|
X\|_{L^\infty_t(L^\infty)}\|\na^2\p_X^{\ell-1}
v\|_{L^{r_{\ell-1}}_t(L^p)}\lesssim J_{\ell-1}^{\ell+2}.
\end{split}
\eeno Resuming the above estimate, the Estimates \eqref{S5eq12} and
\eqref{S5eq13} into \eqref{S5eq15a} and taking $\|a_0\|_{L^\infty}$
to be sufficiently small, we deduce from the maximal regularity
estimate to the heat semi-group that \beq\label{S5eq16}
\begin{split} \|&\p_t\p_X^\ell v\|_{L^{r_\ell}_t(L^p)}+\|\D\p_X^\ell
v\|_{L^{r_\ell}_t(L^p)}+\|\na\p_X^\ell\pi\|_{L^{r_\ell}_t(L^p)}\leq
 \cH_\ell(t)+
C\Bigl(\int_0^t\bigl(\|\na^2 v(t')\|_{L^p}^{r_\ell}\\
 &\qquad\qquad\quad +\|\na
\pi(t')\|_{L^p}^{r_\ell} + (1+\|v(t')\|_{L^p})^{r_\ell} \|\na
v(t')\|_{L^\infty}^{r_\ell}\bigr)\|\na\p_X^{\ell-1}
X(t')\|_{W^{1,p}}^{r_\ell}\,dt'\Bigr)^{\f1{r_\ell}}.
\end{split}
\eeq

 \no$\bullet$ \underline{The estimates of   $\Delta\d_X^{\ell-1} X$}

We first get, by applying the operator $\d_X^{\ell-1}$ to
\eqref{S1eq2}, that
$$
\d_t\d_X^{\ell-1} X+v\cdot\nabla\d_X^{\ell-1} X=\d_X^\ell v,
$$
from which, we deduce, by using $L^p$ type energy estimate and
\eqref{S3eq4a}, that \beno
\begin{split}
\|\p_X^{\ell-1}X(t)\|_{L^p}\leq
&\|\p_{X_0}^{\ell-1}X_0\|_{L^p}+\|\p_X^\ell v\|_{L^1_t(L^p)}\\
\leq
&\|\p_{X_0}^{\ell-1}X_0\|_{L^p}+\|X\|_{L^\infty_t(L^\infty)}\|\na\p_X^{\ell-1}v\|_{L^1_t(L^p)}\leq
J_{\ell-1}^{\ell+1},
\end{split}
\eeno and \beno
\begin{split}
\|\na\p_X^{\ell-1}X(t)\|_{L^p}\leq
&\|\na\p_{X_0}^{\ell-1}X_0\|_{L^p}+\int_0^t\bigl(\|\na
v(t')\|_{L^\infty}\|\na\p_X^{\ell-1}X(t')\|_{L^p}\\
&+\|\na X(t')\|_{L^\infty}\|\na
\p_X^{\ell-1}v(t')\|_{L^p}+\|X(t')\|_{L^\infty}\|\na^2\p_X^{\ell-1}v(t')\|_{L^p}\bigr)\,dt'.
\end{split}
\eeno
Applying Gronwall's inequality and using \eqref{S3eq4a}, we
obtain \beno
\begin{split}
\|\na\p_X^{\ell-1}X\|_{L^\infty_t(L^p)}\leq
&\Bigl(\|\na\p_{X_0}^{\ell-1}X_0\|_{L^p}+\|\na
X\|_{L^\infty_t(L^\infty)}\|\na
\p_X^{\ell-1}v\|_{L^1_t(L^p)}\\
&+\|X\|_{L^\infty_t(L^\infty)}\|\na^2\p_X^{\ell-1}v\|_{L^1_t(L^p)}\Bigr)\exp\bigl(\|\na
v\|_{L^1_t(L^\infty)}\bigr)\lesssim
J_{\ell-1}^{\ell+1}\exp\bigl(CJ_0\bigr).
\end{split}
\eeno This gives \beq\label{S5eq17} \|\d_X^{\ell-1}
X\|_{L^\infty_t(W^{1,p})}
 \leq \cH_\ell(t).
\eeq

By the same manner, noting that
$$
\d_t\Delta \d_X^{\ell-1} X +v\cdot\nabla \Delta \d_X^{\ell-1} X
=-\Delta v\cdot\nabla \d_X^{\ell-1} X -2\nabla
v:\nabla^2\d_X^{\ell-1} X +\Delta\d_X^\ell v,
$$
we get, by using $L^p$ type energy estimate, that \beno
\begin{split}
\|\D\p_X^{\ell-1}X(t)\|_{L^p}\leq
&\|\D\p_{X_0}^{\ell-1}X_0\|_{L^p}+\int_0^t\bigl(\|\D
v(t')\|_{L^p}\|\na\p_X^{\ell-1}X(t')\|_{L^\infty}\\
&+2\|\na
v(t')\|_{L^\infty}\|\na^2\p_X^{\ell-1}X(t')\|_{L^p}
+\|\D\p_X^\ell v(t')\|_{L^p}\bigr)\,dt'.
\end{split}
\eeno
Using \eqref{S3eq2a} and then applying Gronwall's inequality,
we achieve \beq
\begin{split}\label{S5eq18}
\| \D \d_X^{\ell-1} X\|_{L^\infty_t(L^{p})} \leq &\Bigl(
\|\Delta\d_{X_0}^{\ell-1}X_0\|_{L^p}
+\|\na\p_X^{\ell-1}X\|_{L^\infty_t(L^p)}\|\D
v\|_{L^1_t(L^p)}+\|\Delta\d_X^\ell v\|_{L^1_t(L^p)}
\Bigr)\\
&\quad\times\exp\Bigl(C( \|\nabla v\|_{L^1_t(L^\infty)} +\|\Delta
v\|_{L^1_t(L^p)}\bigr)\Bigr)
\\
\leq &\cH_\ell(t)\bigl(1+  \|\Delta\d_X^\ell v\|_{L^1_t(L^p)}\bigr).
\end{split} \eeq
 Substituting \eqref{S5eq16} into \eqref{S5eq18} gives rise to
\beno \begin{split}\|\D\p_X^{\ell-1}X(t)\|_{L^p}\leq
\cH_\ell(t)+\cH_\ell(t)&
\Bigl(\int_0^t\bigl(\|\na^2v(t')\|_{L^p}^{r_\ell}+\|\na\pi(t')\|_{L^p}^{r_\ell}
\\
&+(1+\|v(t')\|_{L^p}^{r_\ell})\|\na v(t')\|_{L^\infty}^{r_\ell}\bigr)\|\D\p_X^{\ell-1}
X(t')\|_{L^p}^{r_\ell}\,dt'\Bigr)^{\f1{r_\ell}}. \end{split} \eeno
Taking $r_\ell$ th power to the above inequality and then applying
Gronwall's inequality yields \beno
\|\D\p_X^{\ell-1}X\|_{L^\infty_t(L^p)}\leq \cH_\ell(t), \eeno which
together with \eqref{S5eq16} and \eqref{S5eq17} ensures that
\beq\label{S5eq19}
\|\p_t\p_X^{\ell}v\|_{L^{r_\ell}_t(L^p)}+\|\D\p_X^\ell
v\|_{L^{r_\ell}_t(L^p)}+\|\na\p_X^\ell\pi\|_{L^{r_\ell}_t(L^p)}+
\|\p_X^{\ell-1}X\|_{L^\infty_t(W^{2,p})}\leq \cH_\ell(t). \eeq
Correspondingly for $s_\ell$ determined by  $\frac{1}{r_\ell}=\frac
12+\frac {1}{s_\ell}$, we have \beq\label{S5eq20} \|\nabla\d_X^\ell
v\|_{L^{r_\ell}_t(L^\infty)\cap
 L^{s_\ell}_t(L^p)}+
\|\d_X^\ell v\|_{L^{s_\ell}_t(L^\infty)}+\|\d_t\d_X^{\ell-1}
X\|_{L^{s_\ell}_t(W^{1,p})} \leq \cH_\ell(t). \eeq This completes
the proof of Proposition \ref{S5prop1}.
\end{proof}

Finally we are in a position to complete the proof of Theorem
\ref{thm1}.

\begin{proof}[Proof of Theorem \ref{thm1}] In order to complete the proof of  Theorem
\ref{thm1}, we first mollify the initial data to be
$(\r_{0,n},v_{0,n}, X_{0,n})$ so that according to \cite{Lions96},
\eqref{inNS} has a unique global solution $(\r_n, v_n, X_n)$ and
then repeating the proofs of Proposition \ref{S2prop1}, Proposition
\ref{S3prop1} and Proposition \ref{S5prop1} to show that the
approximate solutions $(\r_n, v_n, X_n)$ satisfies the uniform
estimates \eqref{S4eq2}, \eqref{S3eq4a} and \eqref{S5eq2}. Finally a
standard compactness argument as that in \cite{Chemin98} can be
applied to prove the limit $(\r, v, X)$ of such approximate
solutions satisfies the required estimate of Theorem \ref{thm1}.
Since the velocity field $v\in L^1_\loc(\R^+; \mbox{Lip}),$ the
uniqueness of such solutions can be proved by using Lagrangian
formulation of \eqref{inNS} as that in \cite{DM12, HPZ}. We omit the
details here.
\end{proof}

 \setcounter{equation}{0}
\section{The proof of \eqref{initial:v0}}\label{Sect7}

The goal of this section is to present the proof of
\eqref{initial:v0}. Similar to Section 2.2 of \cite{Chemin88}, we
define the product operator $R_q$ as follows \beq\label{S6eq5a}
R_q(\alpha_1,\cdots,\alpha_m)\eqdefa\int_{[0,1]^{m}}\int_{\R^2}
\prod_{j=1}^m \alpha_j(x+f_j(\tau)2^{-q}y)\, h(\tau,y)\,dy d\tau,
\eeq where $h\in C([0,1]^m\times\R^2;\cS(\R^2)),$ $f_j\in
L^\infty([0,1]^m)$ and $f_j(\tau)\neq 0$ for $\tau \in ]0,1[^m,$ and
there exists a nonnegative function $H(t,r)\in C([0,1]^m;L^1),$
which is non-increasing  with respect to $r$ variable  so that \beq
\label{S6eq6} |h(t,y)|\leq H(t,|y|). \eeq We also recall the
definition of the maximal function to a $L^1_\loc$ function $f$ from
\cite{BCD} \beno \forall\ x\in\R^2,\quad
Mf(x)\eqdefa\sup_{r>0}\f1{\pi r^2}\int_{|y-x|\leq r}|f(y)|\,dy.
\eeno

 \begin{lem}\label{lem:Rq}
{\sl  Let $X$ be a $\delta,k-$regular vector field (see
\eqref{S6eq4}) and $\al_j$ satisfy
$\mbox{supp}\widehat{\al}_j\subset B(0,2^q),$ for $j=1,\cdots, m.$
 Then for any $\ell\leq k,$ there exist integers $\gamma_n(\la_i)$ such that
 \beq\label{S6eq7}
\begin{split}
 \bigl| (T_{X\cdot\nabla})^\ell& R_q(\alpha_1,\cdots,\alpha_m)
\bigr|\\
 & \lesssim \sum_{n,|\lambda|\leq \ell} 2^{q(\ell-|\lambda|)
\delta } \min_{i=1,\cdots,m} \Bigl\{
  M^{\gamma_n(\la_i)}\bigl( (T_{X\cdot\nabla})^{\lambda_i}\alpha_i \bigr)
\prod_{j=1,j\neq i}^m
\|(T_{X\cdot\nabla})^{\lambda_j}\alpha_j\|_{L^\infty} \Bigr\}.
\end{split} \eeq Here and in what follows, $A\lesssim B$ means that there exists a positive constant
$C$ which depends on $\|X\|^{1-\delta,k-1}_{\delta, X}$ so that $A\leq
CB.$}
\end{lem}

With the above lemma,  similar to Theorem 2.2.1 of \cite{Chemin88},
Theorem 3.2.2 of \cite{Chemin91} and Lemma 2.2 of \cite{ZQ}, we have

\begin{lem}\label{lem:prop}
{\sl If  the vector field $X$ is $\delta,k-$regular, let the
striated norms $\|\cdot\|^{\s,\ell}_{\delta,X},
|\cdot|^{s,p,\ell}_{\delta,X}$ and $|\cdot\|^{s,p,\ell}_{\delta,X}$
be given respectively  by  \eqref{S1eq7}, \eqref{S6eq5} and
\eqref{S1eq3}. Then the following statements hold for all $\ell\leq
k$:
\begin{enumerate}[(I)]
\item if $\mbox{supp}\widehat{v}, \mbox{supp}\widehat{w}\subset B(0,2^q),$
then there exist integers $\gamma_n(\la_1),$ so that
   \beq \label{S6eq8} \bigl|(T_{X\cdot\nabla})^\ell(vw) \bigr|
    \lesssim
    \sum_{n,\lambda_1+\lambda_2\leq \ell}
    2^{q(\ell-\lambda_1-\lambda_2) \delta }
      M^{\gamma_n(\la_1)}\bigl( (T_{X\cdot\nabla})^{\lambda_1}v \bigr)
    \|(T_{X\cdot\nabla})^{\lambda_2}w\|_{L^\infty};
    \eeq

\item for any smooth function $\theta$ with compact support in an annulus $\cC,$ one has
  \beq \label{II1}  \begin{split}
   & \|(T_{X\cdot\nabla})^{\ell} \theta(2^{-q}D)w\|_{L^\infty}
    \lesssim
    2^{ q (\ell\delta -\sigma) }
    \| w\|^{\sigma,\ell}_{\delta,X},
 \\
    &\|(T_{X\cdot\nabla})^\ell \theta(2^{-q}D) w\|_{L^\infty}
    \lesssim
    \sum_{0\leq \lambda\leq \ell}
    2^{ q (\ell-\lambda)\delta   }
    \|(T_{X\cdot\nabla})^\lambda w \| _{L^\infty} ;
    \end{split}\eeq
    and there exist integers $\gamma_n(\la)$  so that
    \beq \label{II2}
    \begin{split}
    &\Bigl\| \bigl( 2^{q(s-\ell\delta)}
         (T_{X\cdot\nabla})^\ell \theta(2^{-q}D)w \bigr)_{\ell^2(\N)}
         \Bigr\|_{L^p}\lesssim |w|^{s,p,\ell}_{\delta,X},
       \\
    &\bigl|(T_{X\cdot\nabla})^\ell \theta(2^{-q}D) w\bigr|
    \lesssim
    \sum_{0\leq n,\lambda\leq \ell}
    2^{ q (\ell-\lambda)\delta   }
    M^{\gamma_n(\la)}\bigl( (T_{X\cdot\nabla})^\lambda w \bigr);
    \end{split}\eeq

\item let $\chi(\xi)$ be smooth $S^m$ Fourier-multiplier with $\chi(L\xi)=L^m\chi(\xi)$ for $|\xi|\geq M,$  then
    \beq \|\chi(D)w\|^{\sigma-m,\ell}_{\delta,X}
    \lesssim \| w\|^{\sigma ,\ell}_{\delta,X},
   \andf |\chi(D)w |^{s-m,p,\ell}_{\delta,X}
     \lesssim  | w |^{s,p,\ell}_{\delta,X} ;\label{III2}
    \eeq

\item We also have the  following estimates for the paraproduct and the remainder term  to the Bony's decomposition:
    \beq\label{IV1} \|T_v w\|^{\sigma_{1}+\sigma_2,\ell}_{\delta,X}
    \lesssim
    \|v\|^{\sigma_1,\ell}_{\delta,X}
    \,\|w\|^{\sigma_2,\ell}_{\delta,X}
    \ \mbox{if}\ \sigma_1< 0,
    \ \  \|T_v w\|^{\sigma_{2},\ell}_{\delta,X}
    \lesssim
    \bigl( \|v\|_{L^\infty}+\|v\|^{0,\ell}_{\delta,X}\bigr)
    \,\|w\|^{\sigma_2,\ell}_{\delta,X},
    \eeq
    and
    \beq  \label{IV2} |T_v w |^{s+\sigma,p,\ell}_{\delta,X}
    \lesssim
    \|v\|^{\sigma,\ell}_{\delta,X}
     \,|w |^{s,p,\ell}_{\delta,X}\
    \ \mbox{if}\   \sigma< 0,\ \ \  |T_v w |^{s,p,\ell}_{\delta,X}
    \lesssim
    \bigl( \|v\|_{L^\infty}+\|v\|^{0,\ell}_{\delta,X}\bigr)
    \, |w |^{s,p,\ell}_{\delta,X},
    \eeq
    or for any $\epsilon>0,$
    \beq \label{IV3}
|T_v w |^{ \sigma+s,p,\ell}_{\delta,X}
    \lesssim
     |v |^{s,p,\ell}_{\delta,X}
     \,\|w \|^{\sigma,\ell}_{\delta,X} \ \
   \ \mbox{if}\       \ s< 0,
     \ \  |T_v w |^{\sigma,p,\ell}_{\delta,X}
    \lesssim
    \bigl(  \|v \|_{W^{\epsilon,p}}+ |v |^{0,p,\ell}_{\delta,X}\bigr)
    \, \|w \|^{\sigma,\ell}_{\delta,X}, \eeq
    and for $\sigma_1+\sigma_2>\ell\delta $, one has
    \beq \label{IV4} \|R(v,w)\|^{\sigma_1+\sigma_2,\ell}_{\delta,X}
     \lesssim
     \|v\|^{\sigma_1,\ell}_{\delta,X}
     \|w\|^{\sigma_2,\ell}_{\delta,X},
     \quad |R(v,w) |^{\sigma_1+\sigma_2,p,\ell}_{\delta,X}
     \lesssim
     \|v\|^{\sigma_1,\ell}_{\delta,X}
      |w |^{\sigma_2,p,\ell}_{\delta,X} ;
  \eeq

\item If $\sigma\in ]\ell\delta,1[$ (or $\sigma\in ]\ell\delta -1,1[$ if $\div X=0$),
then one
 \beq \|w\|^{\sigma,\ell}_{\delta,X}
    \sim \sum_{0\leq i\leq \ell}\|\d_X^i \,w\|_{C^{\sigma-i \delta}},
   \andf |w |^{\sigma,p,\ell}_{\delta,X}
    \sim |w \|^{\sigma,p,\ell}_{\delta,X}.
    \label{V2}
    \eeq
\end{enumerate} }
\end{lem}
The proofs of Lemma \ref{lem:Rq} and Lemma \ref{lem:prop} will be
postponed to the Appendix \ref{appb}. We mention that the main
difference between the proof here and that in \cite{Chemin88,
Chemin91, ZQ} is that the Sobolev space $W^{s,p}$ is of
Triebel-Lizorkin type space, which is not of Besov type spaces as in
the previous papers.

We now present the proof of \eqref{initial:v0}.

\begin{proof}[Proof of \eqref{initial:v0}] It is easy to observe
from \eqref{initial:v0} and \eqref{V2} that it suffices to prove
\beq\label{S6eq18} |\p_{X_0}v_0|^{1-\e/{k},p,k-1}_{\e/{k},X_0}\leq
C\bigl(\|\om_0\|_{L^1\cap L^p}+|\om_0|^{0,p,k}_{\e/{k},X_0}\bigr).
\eeq As a matter of fact, in view of \eqref{initial:X0} and $p>2$,
we have
\begin{equation*}
X_0\in   C^{1-\delta}_{\delta}(X_0, k-1),\quad \mbox{for any}\
\delta \in ]0,1[ \with  \delta k<1.
\end{equation*}
In particular, this gives \beq\label{S6eq19} X_0\in
C^{1-\varepsilon/k}_{\varepsilon/k}(X_0,k-1) \quad \mbox{for
any}\quad \varepsilon\in ]0,1[. \eeq While for any smooth cut-off
function $\zeta(\xi)= \left\{\begin{array}{l} \displaystyle
1, \quad |\xi|\leq 1,\\
\displaystyle 0,\quad |\xi|\geq 2, \end{array}\right.$, we deduce
from \eqref{initial:X0} and \eqref{S6eq0}  that \beq\label{S6eq20}
|\p_{X_0}\zeta(D)v_0|^{1-\e/{k},p,k-1}_{\e/{k},X_0}\leq
C\|\om_0\|_{L^1\cap L^p}. \eeq  So it remains to estimate
$\p_{X_0}(Id-\zeta(D))\na^\perp\D^{-1}\om_0.$ Indeed
 applying Bony's decomposition gives
$$\longformule{
\p_{X_0}(Id-\zeta(D))\na^\perp\D^{-1}\om_0=T_{X_0\cdot\na}(Id-\zeta(D))\na^\perp\D^{-1}\om_0}{{}+
T_{\na(Id-\zeta(D))\na^\perp\D^{-1}\om_0}\cdot
X_0+R(X_0,\na(Id-\zeta(D))\na^\perp\D^{-1}\om_0).} $$ Applying
\eqref{III2} gives \beno
\bigl|T_{X_0\cdot\na}(Id-\zeta(D))\na^\perp\D^{-1}\om_0\bigr|^{1-\e/{k},p,k-1}_{\e/{k},X_0}\leq
\bigl|(Id-\zeta(D))\na^\perp\D^{-1}\om_0\bigr|^{1,p,k}_{\e/{k},X_0}\leq
C|\om_0|^{0,p,k}_{\e/{k},X_0}. \eeno Applying \eqref{IV3} and the
property \eqref{S6eq19}   with $0<\varepsilon'<\varepsilon$ yields
\begin{align*}
\bigl|T_{\na(Id-\chi(D))\na^\perp\D^{-1}\om_0}\cdot
X_0\bigr|^{1-\e/{k},p,k-1}_{\e/{k},X_0} &\leq
C|\om_0|^{(\varepsilon'-\varepsilon)/k,p,k-1}_{\e/{k},X_0}
\|X_0\|^{1-\varepsilon'/{k},k-1}_{\e/{k},X_0}
\\
&\leq C|\om_0|^{0,p,k-1}_{\e/{k},X_0}
 \|X_0\|^{1-\varepsilon'/{k},k-1}_{\varepsilon'/{k},X_0}.
\end{align*}  Finally applying \eqref{IV4} and then using \eqref{III2} leads
to \beno\begin{split} \bigl|
R(X_0,&\na(Id-\zeta(D))\na^\perp\D^{-1}\om_0)\bigr|^{1-\e/{k},p,k-1}_{\e/{k},X_0}\\
\leq
&
C\bigl|\na(Id-\zeta(D))\na^\perp\D^{-1}\om_0\bigr|^{0,p,k-1}_{\e/{k},X_0}\|X_0\|^{1-\e/{k},k-1}_{\e/{k},X_0}\\
\leq& C
|\om_0|^{0,p,k}_{\e/{k},X_0}\|X_0\|^{1-\e/{k},k-1}_{\e/{k},X_0}.
\end{split}
\eeno This in turn shows that \beno
\bigl|\p_{X_0}(Id-\zeta(D))\na^\perp\D^{-1}\om_0
\bigr|^{1-\e/{k},p,k-1}_{\e/{k},X_0}\leq C
|\om_0|^{0,p,k}_{\e/{k},X_0}\|X_0\|^{1-\e/{k},k-1}_{\e/{k},X_0},
\eeno which together with \eqref{S6eq20} gives rise to
\eqref{S6eq18}. \end{proof}

\appendix

\setcounter{equation}{0}
\section{The proofs of Lemma \ref{lem:Rq} and Lemma
\ref{lem:prop}}\label{appb}

In this appendix we shall present  the proofs of Lemma \ref{lem:Rq}
and Lemma \ref{lem:prop}.

\begin{proof}[Proof of Lemma \ref{lem:Rq}]
We first prove \eqref{S6eq7} for $\ell=0.$ In this case, according
to \eqref{S6eq5a} and \eqref{S6eq6}, for any $i\in \{1,\cdots m\
\},$ we have \beq\label{Apeq13} |R_q(\alpha_1,\cdots,\alpha_m)|\leq
\prod_{j=1,j\neq
i}^m\|\al_j\|_{L^\infty}\int_{[0,1]^m}\int_{\R^2}|\al_i(x
+f_i(\tau)2^{-q}y)|H(\tau,|y|)\,dy\,d\tau. \eeq Yet similar to the
proof of Proposition 1.16 of \cite{BCD}, for any $\tau\in ]0,1[,$
one has \beno
\begin{split}
\int_{\R^2}&|\al_i(x+f_i(\tau)2^{-q}y)|H(\tau,|y|)\,dy\\
=&\Bigl(\f{2^{q}}{f_i(\tau)}\Bigr)^2\int_0^\infty
H\Bigl(\tau,\f{2^q r}{f_i(\tau)}\Bigr)\int_{|y|=r}|\al_i(x-y)|\,d\s_r(y)\,dr\\
=&-\Bigl(\f{2^{q}}{f_i(\tau)}\Bigr)^3\int_0^\infty \p_r
H\Bigl(\tau,\f{2^q
r}{f_i(\tau)}\Bigr)\int_{B(x,r)}|\al_i(y)|\,dy\,dr\\
\leq &-\pi\Bigl(\f{2^{q}}{f_i(\tau)}\Bigr)^3
M(\al_i)(x)\int_0^\infty
\p_r H\Bigl(\tau,\f{2^q r}{f_i(\tau)}\Bigr)r^2\,dr\\
=&M(\al_i)(x)\|H(\tau,\cdot)\|_{L^1}.
\end{split}
\eeno Substituting the above estimate into \eqref{Apeq13} shows that
\eqref{II1} holds for $\ell=0.$

Inductively, we assume that \eqref{S6eq7}, which implies
\eqref{S6eq8},  and \eqref{II2}, which we shall prove late,  hold
for $\ell\leq k-1,$ we are going to prove that \eqref{S6eq7} holds
for $\ell+1.$ Note for $ \varphi_1(\xi)=i\xi\varphi(\xi),$ one has
$\sum_{q_1\in\Z}2^{q_1}\varphi_1(2^{-q_1}D)=\na,$ which gives (see
also the Appendix of \cite{Chemin88})
\begin{align*}
T_{X\cdot\nabla}&R_q(\alpha_1,\cdots,\alpha_m)\\
 &=\sum_{q_1\leq
q+N_0}2^{ q_1} S_{q_1-1}X\cdot  \varphi_1(2^{-q_1}D)
R_q(\alpha_1,\cdots,\alpha_m)
\\
&=\sum_{q_1\leq q+N_0}2^{ q_1} \bigl( S_{q_1-1}X- S_{q-1}X \bigr)
\cdot  \varphi_1(2^{-q_1}D) R_q(\alpha_1,\cdots,\alpha_m)
\\
&\quad +\sum_{q_1\leq q+N_0} 2^{ q_1} \bigl(
S_{q-1}X-S_{q_1-1}X\bigr) \cdot\sum_{j=1}^m
 R_q(\alpha_1,\cdots, \varphi_1(2^{-q_1}D)\alpha_j,\cdots,\alpha_m)
 \\
 &\quad
 +\sum_{q_1\leq q+N_0}
 S_{q_1-1}X \cdot \sum_{j=1}^m
 R_q(\alpha_1,\cdots,\nabla\varphi(2^{-q_1}D)\alpha_j,\cdots,\alpha_m)
 \\
 &\eqdefa R^1_q+R^2_q+R^3_q.
\end{align*}
It follows from the inductive assumption, \eqref{S6eq8}, that
\begin{align*}
| (T_{X\cdot\nabla})^\ell R^1_q| \lesssim & \sum_{\substack{q_1\leq
q+N_0\\ n_1,\ell_1+\ell_2\leq \ell}} 2^{q_1}
2^{q\delta(\ell-\ell_1-\ell_2)} \|(T_{X\cdot\nabla})^{\ell_1}
(S_{q_1-1}X-S_{q-1}X)\|_{L^\infty}
\\
&\qquad\qquad\qquad\qquad \times M^{\gamma_{n_1}(\ell_2)}\bigl(
(T_{X\cdot\nabla})^{\ell_2} \varphi_1(2^{-q_1}D)
R_q(\alpha_1,\cdots,\alpha_m) \bigr).
\end{align*}
For $q_1< q $,  applying  (\ref{II1})  and using assumption that
$1-k\delta>0$ yields
\begin{align*}
\|(T_{X\cdot\nabla})^{\ell_1} (S_{q_1-1}X-S_{q-1}X)\|_{L^\infty}
\lesssim& \sum_{q_2=q_1-1}^{q-2} \|(T_{X\cdot\nabla})^{\ell_1}
\Delta_{q_2}X\|_{L^\infty}
\\
\lesssim& \sum_{q_2=q_1-1}^{q-2} 2^{-q_2(1-(1+\ell_1)\delta)}
\|X\|^{1-\delta,k-1}_{\delta,X}\\
 \lesssim &
2^{-q_1(1-(1+\ell_1)\delta)}\|X\|^{1-\delta,k-1}_{\delta,X}.
\end{align*}
Along the same line for $q_1\in [q,q+N_0]$, there holds the same
estimate for $\|(T_{X\cdot\nabla})^{\ell_1}
(S_{q_1-1}X-S_{q-1}X)\|_{L^\infty}.$

Whereas applying \eqref{II2} and then  the inductive assumption
gives
\begin{align*}
&\Bigl| (T_{X\cdot\nabla})^{\ell_2}
 \varphi_1(2^{-q_1}D)R_q(\alpha_1,\cdots,\alpha_m) \Bigr|
\\
&\lesssim \sum_{0\leq  n_2,\ell_3\leq
\ell_2}2^{q_1\delta(\ell_2-\ell_3)} M^{\gamma_{n_2}(\ell_3)}\bigl(
(T_{X\cdot\nabla})^{\ell_3} R_q(\alpha_1,\cdots,\alpha_m)\bigr)
\\
&\lesssim \sum_{\substack{|\lambda|\leq \ell_3\leq \ell_2\\ n_3\leq
\ell_3}} 2^{q(\ell_2-|\lambda|) \delta } \min_{i=1,\cdots,m} \Bigl\{
  M^{\gamma_{n_2}(\ell_3)+\gamma_{n_3}(\la_i)}\bigl( (T_{X\cdot\nabla})^{\lambda_i}\alpha_i \bigr)
\prod_{j=1,j\neq i}^m
\|(T_{X\cdot\nabla})^{\lambda_j}\alpha_j\|_{L^\infty} \Bigr\}.
\end{align*}
Hence we obtain
\begin{align}\label{R1q}
| (T_{X\cdot\nabla})^\ell R^1_q | \lesssim \sum_{n,|\lambda|\leq
\ell } 2^{q\delta(\ell+1-|\lambda|)} \min_{i=1,\cdots,m} \Bigl\{
  M^{\gamma_n(\la_i)}\bigl( (T_{X\cdot\nabla})^{\lambda_i}\alpha_i \bigr)
\prod_{j\neq i}
\|(T_{X\cdot\nabla})^{\lambda_j}\alpha_j\|_{L^\infty} \Bigr\}.
\end{align}
The same argument leads to the same estimate  for $R^2_q$.

Finally  noticing that
$$
S_{q_1-1}X(x) =S_{q_1-1}X(x+2^{-q}f_j(t)y) -\int^1_0
2^{-q}f_j(t)y\cdot \nabla S_{q_1-1} X (x+\tau 2^{-q}f_j(t)
y)\,d\tau,
$$
we write \beq\label{R3q}
\begin{split}
 R^3_q=&\sum_{j=1}^m
 R_q(\alpha_1,\cdots,T_{X\cdot\nabla}\alpha_j,\cdots,\alpha_m)
 \\
& -\sum_{q_1\leq q+N_0} 2^{{q_1-q}}
 \sum_{j=1}^m
 R_q^{(j)}(\nabla S_{q_1-1}X,
 \alpha_1,\cdots, \varphi_1(2^{-q_1}D)\alpha_j,\cdots,\alpha_m),
\end{split} \eeq
where  $R_q^{(j)}$ is given by \beno
\begin{split}
 &R_q^{(j)}(\nabla S_{q_1-1}X,
 \alpha_1,\cdots,
 \varphi_1(2^{-q_1}D)\alpha_j,\cdots,\alpha_m)\\
 &=\int_0^1\int_{[0,1]^m}\int_{\R^d}\na
 S_{q-1}X(x+\tau 2^{-q}f_j(t)y)\\
 &\quad\times \varphi_1(2^{-q}D)\al_j(x+2^{-q}f_j(t)y)\Prod_{i\neq
 j}\al_i(x+2^{-q}f_i(t)y)\cdot y f_j(t) h(t,y)
\,dy\,dt\,d\tau.
\end{split}\eeno
Using the inductive assumption and taking into account the fact that
$$
\|(T_{X\cdot\nabla})^{\mu}\nabla S_{q_1-1}X\|_{L^\infty} \lesssim
\sum_{q_2\leq q_1-2} 2^{q_2\delta(1+\mu)} \|\nabla X\|^{
-\delta,\mu}_{\delta,X} \lesssim
2^{q_1\delta(1+\mu)}\|X\|^{1-\delta,k-1}_{\delta,X},
\quad\mu\leq\ell\leq k-1,
$$
we write \beno
\begin{split}
\bigl|&(T_{X\cdot\na})^\ell R_q^{(j)}(\nabla S_{q_1-1}X,
 \alpha_1,\cdots,
 \varphi_1(2^{-q_1}D)\alpha_j,\cdots,\alpha_m)\bigl|\\
 &\lesssim \sum_{n,\mu+|\la|\leq
 \ell}2^{q(\ell-\mu-|\la|)\delta}\|(T_{X\cdot\na})^{\mu}\nabla
 S_{q_1-1}X\|_{L^\infty}\\
 &\qquad\qquad\qquad\times\min_{i=1,\cdots,m}\Bigl\{M^{\gamma_n(\la_i)}\bigl( (T_{X\cdot\nabla})^{\lambda_i}\bar{\alpha}_i \bigr)
\prod_{j=1,j\neq i}^m
\|(T_{X\cdot\nabla})^{\lambda_j}\bar{\alpha}_j\|_{L^\infty}
\Bigr\}\\
&\lesssim \sum_{n,|\la|\leq
 \ell}2^{q(\ell+1-|\lambda|)\delta}\|X\|^{1-\delta,k-1}_{\delta,X}\min_{i=1,\cdots,m}\Bigl\{M^{\gamma_n(\la_i)}\bigl( (T_{X\cdot\nabla})^{\lambda_i}{\alpha}_i \bigr)
\prod_{j=1,j\neq i}^m
\|(T_{X\cdot\nabla})^{\lambda_j}{\alpha}_j\|_{L^\infty}\Bigr\}
\end{split}
\eeno where $\bar{\alpha}_i= \left\{\begin{array}{l} \displaystyle
\alpha_i, \qquad\qquad\quad\ \  i\neq j,\\
\displaystyle \varphi_1(2^{-q_1}D)\alpha_j ,\quad i=j,
\end{array}\right.$ and we used \eqref{II1} and \eqref{II2} in the
last step. This together with \eqref{R3q} and the inductive
assumption ensures that
 that $R^3_q$  shares the same estimate as $R^1_q$ in \eqref{R1q}. This
completes the proof of Lemma \ref{lem:Rq}.
\end{proof}

Next let us turn to the proof of  Lemma \ref{lem:prop}.

\begin{proof}[Proof of Lemma \ref{lem:prop}]
The statements concerning the norm
$\|\cdot\|^{\sigma,\ell}_{\delta,X}$ have been proved in Section 2.2
of \cite{Chemin88}, Theorem 3.2.2 of \cite{Chemin91} and Lemma 2.2
of \cite{ZQ}. We just need to prove the  estimates for the norm
$|\cdot|^{\sigma,p,\ell}_{\delta,X}$. Firstly the Estimate
\eqref{S6eq8} follows directly from \eqref{S6eq7}.

 \no$\bullet$ \underline{Proof of (\ref{II2})}

 In view of Proposition 1.16 and Remark 1.17 of \cite{BCD}, \eqref{II2} holds for $\ell=0$. Inductively, let us assume  that \eqref{II2}
holds for $\ell\leq k-1,$ we are going to prove this estimate for
$\ell+1.$ Notice that
\begin{align*}
(T_{X\cdot\nabla})^{\ell+1}\theta(2^{-q}D)(w)
=(T_{X\cdot\nabla})^\ell[T_{X\cdot\nabla};\theta(2^{-q}D)]w
+(T_{X\cdot\nabla})^\ell \theta(2^{-q}D)(T_{X\cdot\nabla}w).
\end{align*}
It is easy to observe from the spectral properties to terms of
Bony's decomposition that there exist a smooth function
$\tilde{\th}$ with compact support in some annulus of $\R^2$ and a
fixed integer $N_\th$ so that
\begin{align*}
[T_{X\cdot\nabla};\theta(2^{-q}D)]w &=\sum_{|q_1-q|\leq N_\theta}
[S_{q_1-1}X; \theta(2^{-q}D)] \varphi(2^{-q_1}D) \na\tilde
\theta(2^{-q}D) w
\\
=&2^q\sum_{|q_1-q|\leq N_\theta}
 \bigl(S_{q_1-1}X-S_{q-1}X\bigr)
\cdot\theta(2^{-q}D)\varphi(2^{-q_1}D)\tilde\th_1(2^{-q} D) w
\\
& +2^q[S_{q-1}X;\theta(2^{-q}D)] \tilde\th_1(2^{-q} D) w
\\
&-2^q\sum_{|q_1-q|\leq N_\theta}\theta(2^{-q}D)
\bigl(\bigl(S_{q_1-1}X-S_{q-1}X\bigr) \cdot
\varphi(2^{-q_1}D)\tilde\th_1(2^{-q} D) w\bigr)
\\ \eqdefa
&T^1_q+T^2_q+T^3_q,
\end{align*}
where $\tilde\th_1(\xi)\eqdefa i\xi\tilde\th(\xi).$

It follows form (\ref{S6eq8}) and (\ref{II1})  that
\begin{align*}
|(T_{X\cdot\nabla})^\ell T^1_q| &\lesssim
\sum_{n_1,\lambda_1+\lambda_2\leq
\ell}2^{q\delta(\ell-\lambda_1-\lambda_2)}
 2^q\sum_{|q_1-q|\leq N_\theta}
\|(T_{X\cdot\nabla})^{\lambda_1} (S_{q_1-1}X-S_{q-1}X)\|_{L^\infty}
\\
&\qquad\qquad\qquad\qquad\qquad\qquad \times
M^{\gamma_{n_1}(\la_2)}\bigl( (T_{X\cdot\nabla})^{\lambda_2}
  \theta(2^{-q}D)\varphi(2^{-q_1}D)\tilde\th_1(2^{-q}D) w \bigr)
  \\
  &\lesssim
\sum_{n_1,\lambda_1+\lambda_2\leq
\ell}2^{q\delta(\ell-\lambda_1-\lambda_2)} \,2^q\, 2^{q(\lambda_1
\delta-(1-\delta))} \|X\|^{1-\delta,k-1}_{\delta,X}
\\
&\qquad\qquad\qquad\qquad\qquad\qquad \times
  M^{\gamma_{n_1}(\la_2)}\bigl( (T_{X\cdot\nabla})^{\lambda_2}
  \theta(2^{-q}D)\varphi(2^{-q_1}D)\tilde\th_1(2^{-q}D) w \bigr),
\end{align*}
which together with the inductive assumption implies
\beq\label{T1q2} \bigl| (T_{X\cdot\nabla})^\ell T^1_q \bigr|
  \lesssim
\sum_{n,\lambda\leq \ell} 2^{q\delta(\ell+1-\lambda)}
  M^{\gamma_n(\la)}\bigl( (T_{X\cdot\nabla})^{\lambda}
  w \bigr).\eeq
Furthermore, there holds \beq\label{T1q}
\begin{split}
\Bigl\|\bigl(& 2^{q(s-(\ell+1)\delta)} (T_{X\cdot\nabla})^\ell T^1_q
\bigr)_{\ell^2(\N)} \Bigr\|_{L^p}\\
 &\lesssim \sum_{\lambda\leq
\ell} \Bigl\| \bigl( 2^{q(s-\lambda\delta)}
  M^{\gamma(\la)}
  \bigl( (T_{X\cdot\nabla})^{\lambda}
  \theta(2^{-q}D)  w\bigr) \bigr)_{\ell^2(\N)}
\Bigr\|_{L^p}
\\
&\lesssim \sum_{\lambda\leq \ell} \Bigl\| \bigl(
2^{q(s-\lambda\delta)}
  (T_{X\cdot\nabla})^{\lambda}
  \theta(2^{-q}D)  w \bigr)_{\ell^2(\N)}
\Bigr\|_{L^p} \lesssim |w|^{s,p,\ell}_{\delta,X}.
\end{split}\eeq

Whereas by applying  (\ref{II2}) for $\la\leq \ell$ and
\eqref{S6eq8}, one has
\begin{align*}
|(T_{X\cdot\nabla})^\ell T^3_q| &\lesssim
 2^q\sum_{\substack{|q_1-q|\leq N_\theta\\ n_1,
\la\leq \ell}} 2^{q\delta(\ell-\la)} M^{\gamma_{n_1}(\la)}\bigl(
(T_{X\cdot\nabla})^{\la} \bigl( (S_{q_1-1}X-S_{q-1}X)
\varphi(2^{-q_1}D)\tilde\th_1(2^{-q}D) w\bigr) \bigr)
  \\
  &\lesssim
\sum_{\substack{|q_1-q|\leq N_\theta, n_1\leq \ell\\
n_2,\lambda_1+\lambda_2\leq\la\leq
\ell}}2^{q\delta(\ell-\lambda_1-\lambda_2)} \,2^q\,
\|(T_{X\cdot\nabla})^{\lambda_1}(S_{q_1-1}X-S_{q-1}X)\|_{L^\infty}
\\
&\qquad\qquad\qquad\qquad\qquad\qquad\times
M^{\gamma_{n_1}(\la)+\gamma_{n_2}(\la_2)}\bigl(
(T_{X\cdot\nabla})^{\lambda_2}
  \varphi(2^{-q_1}D)\tilde\th_1(2^{-q}D) w \bigr).
\end{align*}
Then the  Estimates \eqref{T1q2}-\eqref{T1q}  also hold  for
$T^3_q$.

To deal with $T^2_q,$ let us denote $h_\th\eqdefa\cF^{-1}\th,$ then
one has
\begin{align*}
T^2_q &= 2^{3q} \int_{\R^2}h_\theta(2^{q}(x-y)) \bigl( S_{q
-1}X(x)-S_{q -1}X(y)\bigr) \cdot \tilde\th_1(2^{-q}D) w(y)\,dy
\\
&=  \int_{\R^2}\int^1_0
 h_\theta(z)
z\cdot\nabla S_{q -1}X(x+(1-t)2^{-q}z) \cdot\tilde\th_1(2^{-q}D)
w(x-2^{-q}z)\,dt\,dz,
\end{align*}
so that by applying \eqref{S6eq7}, we obtain
\begin{align*}
|(T_{X\cdot\nabla})^\ell T^2_q| \lesssim
\sum_{n,\lambda_1+\lambda_2\leq
\ell}2^{q\delta(\ell-\lambda_1-\lambda_2)}
\|(T_{X\cdot\nabla})^{\lambda_1}(\nabla S_{q-1} X)\|_{L^\infty}
M^{\gamma_n(\la_2)}\bigl((T_{X\cdot\nabla})^{\lambda_2}
\tilde\th_1(2^{-q}D) w\bigr).
\end{align*}
Yet applying \eqref{II1} yields
$$
\|(T_{X\cdot\nabla})^{\lambda_1}(\nabla S_{q-1} X)\|_{L^\infty}
\lesssim 2^{ q\delta(1+\lambda_1)} \|X\|^{1-\delta,k-1}_{\delta,X},
$$
from which and  the inductive assumption, we conclude that the
Estimates \eqref{T1q2}-\eqref{T1q} hold for $T^2_q$ as well.

 Finally note that
$$
|T_{X\cdot\nabla} w|^{s-\delta,p, \ell}_{\delta,X} \lesssim
|w|^{s,p,\ell+1}_{\delta,X},
$$
we deduce from the inductive assumption that \beno
\begin{split}
&\bigl|(T_{X\cdot\nabla})^\ell
\theta(2^{-q}D)(T_{X\cdot\nabla}w)\bigr|\lesssim
\sum_{n,\la\leq \ell}2^{q\bigl((\ell+1)-(\la+1)\bigr)}M^{\gamma_n(\la+1)}\bigl((T_{X\cdot\na})^{\la+1}w\bigr)\andf\\
&\Bigl\|\Bigl(
2^{q\bigl((s-\delta)-\ell\delta\bigr)}(T_{X\cdot\nabla})^\ell
\theta(2^{-q}D)(T_{X\cdot\nabla}w) \Bigr)_{\ell^2(\N)}
\Bigr\|_{L^p}\lesssim |w|^{s,p,\ell+1}_{\delta,X}.
\end{split}
\eeno This proves (\ref{II2}) holds for $\ell+1$.

 \no$\bullet$ \underline{Proof of (\ref{III2})} Without loss of
 generality, we may assume that $\mbox{supp}\chi\subset \{x:\ |\xi |\geq M\
 \}.$ Then
it is easy to observe that there exists a smooth function
$\varphi_\chi$ with compact support in some annulus of $\R^2$ such
that for $q\geq 0$
\begin{align*}
  \Delta_q (T_{X\cdot\nabla})^\ell \chi(D)w
&=  2^{qm}\Delta_q (T_{X\cdot\nabla})^\ell \varphi_\chi(2^{-q}D) w.
\end{align*}
Then we infer
\begin{align*}
\|(T_{X\cdot\nabla})^\ell \chi(D)w\|_{W^{s-m-\ell\delta,p}}
&=\Bigl\| \bigl( 2^{q(s-m-\ell\delta)} \Delta_q
(T_{X\cdot\nabla})^\ell \chi(D)w \bigr)_{\ell^2(\N)} \Bigr\|_{L^p}
\\
&= \Bigl\| \bigl( 2^{q(s -\ell\delta)} \Delta_q
(T_{X\cdot\nabla})^\ell \varphi_\chi(2^{-q}D) w \bigr)_{\ell^2(\N)}
\Bigr\|_{L^p}
\\
&\lesssim \Bigl\| \bigl( 2^{q(s -\ell\delta)} M \bigl(
(T_{X\cdot\nabla})^\ell \varphi_\chi(2^{-q}D) w
\bigr)\bigr)_{\ell^2(\N)} \Bigr\|_{L^p}
\\
&\lesssim \Bigl\| \bigl( 2^{q(s -\ell\delta)}
 (T_{X\cdot\nabla})^\ell
\varphi_\chi(2^{-q}D) w \bigr)_{\ell^2(\N)} \Bigr\|_{L^p}
\\
&\lesssim |w|^{s,p,\ell}_{\delta,X}, \hbox{ by  (\ref{II2})}.
\end{align*}
And (\ref{III2}) follows.

 \no$\bullet$ \underline{Proof of the para-product estimates.}
 Due to the support properties to the Fourier transform of the terms in the  para-product decomposition, for any
  integer $\ell,$ there exists a positive integer $N_\ell$ so that
 \begin{align*}
 \Delta_q (T_{X\cdot\nabla})^\ell T_v w
 =\Delta_q \sum_{|q_1-q|\leq N_\ell} (T_{X\cdot\nabla})^\ell (S_{q_1-1}v \Delta_{q_1}w).
 \end{align*}
Applying \eqref{S6eq8} gives
 \begin{align*}
 \Bigl| \Delta_q (T_{X\cdot\nabla})^\ell T_v w \Bigr|
 &\lesssim
 \sum_{|q_1-q|\leq N}
  M\bigl( (T_{X\cdot\nabla})^\ell (S_{q_1-1}v \Delta_{q_1}w) \bigr)
  \\
  &\lesssim
 \sum_{\substack{|q_1-q|\leq N\\ n,\lambda_1+\lambda_2\leq \ell}}
  2^{q_1\delta(\ell-\lambda_1-\lambda_2)}
  \|(T_{X\cdot\nabla})^{\lambda_1} S_{q_1-1}v\|_{L^\infty}
  M^{\gamma_n(\la_2)+1}\bigl( (T_{X\cdot\nabla})^{\lambda_2}\Delta_{q_1}w \bigr)  .
 \end{align*}
 Yet it follows from \eqref{II1} that
 \begin{equation*}
 \|(T_{X\cdot\nabla})^{\lambda_1} S_{q_1-1}v\|_{L^\infty}
 \lesssim\left\{\begin{array}{l}
  2^{q_1(-\sigma+\lambda_1\delta)}\|v\|^{\sigma,\ell}_{\delta,X}\
 \hbox{ if }\ \sigma<0,
 \\
 \|v\|_{L^\infty}\
 \hbox{ if }\lambda_1=0,
 \\
 2^{q_1\lambda_1\delta}\|v\|^{0,k}_{\delta,X}\
 \hbox{ if }\lambda_1>0.
 \end{array}\right.
 \end{equation*}
Therefore, we get, by applying \eqref{II2} that  for $\sigma<0$
\begin{align*}
\|(T_{X\cdot\nabla})^\ell T_v w\|_{W^{s+\sigma-\ell\delta,p}}
&=\Bigl\| \bigl( 2^{q(s+\sigma-\ell\delta)} \Delta_q
(T_{X\cdot\nabla})^\ell T_v w \bigr)_{\ell^2(\N)} \Bigr\|_{L^p}
\\
&\lesssim \|v\|^{\sigma,\ell}_{\delta,X} \sum_{\la_2\leq\ell}\Bigl\|
\bigl( 2^{q(s-\la_2\delta)}
(T_{X\cdot\nabla})^{\lambda_2}\Delta_{q_1}w \bigr)_{\ell^2(\N)}
\Bigr\|_{L^p}
\\
&\lesssim \|v\|^{\sigma,\ell}_{\delta,X} |w|^{s,p,\ell}_{\delta,X},
\end{align*}
and
\begin{align*}
\|(T_{X\cdot\nabla})^\ell T_v w\|_{W^{s -\ell\delta,p}} &=\Bigl\|
\bigl( 2^{q(s -\ell\delta)} \Delta_q  (T_{X\cdot\nabla})^\ell T_v w
\bigr)_{\ell^2(\N)} \Bigr\|_{L^p}
\\
&\lesssim
\bigl(\|v\|_{L^\infty}+\|v\|^{0,\ell}_{\delta,X}\bigr)\sum_{\la_2\leq\ell}\Bigl\|
\bigl( 2^{q(s-\la_2\delta)}
(T_{X\cdot\nabla})^{\lambda_2}\Delta_{q_1}w \bigr)_{\ell^2(\N)}
\Bigr\|_{L^p}
\\
&\lesssim (\|v\|_{L^\infty}+\|v\|^{0,\ell}_{\delta,X})
|w|^{s,p,\ell}_{\delta,X}.
\end{align*}
This proves \eqref{IV2}.

By the same manner, we have for $s<0$
\begin{align*}
\|(T_{X\cdot\nabla})^\ell T_v w\|_{W^{\sigma+s -\ell\delta,p}}
&=\Bigl\| \bigl( 2^{q(\sigma+s-\ell\delta)} \Delta_q
(T_{X\cdot\nabla})^\ell T_v w \bigr)_{\ell^2(\N)} \Bigr\|_{L^p}
\\
 &\lesssim \|w\|^{\sigma,\ell}_{\delta,X} \sum_{ \lambda_1 \leq \ell}
\Bigl\| \Bigl( \sum_{q_2\leq q+N} 2^{(q-q_2)( s-\lambda_1\delta)}
2^{ q_2 ( s-\lambda_1\delta)} (T_{X\cdot\nabla})^{\lambda_1}
\Delta_{q_2} v \Bigr)_{\ell^2(\N)} \Bigr\|_{L^p}
\\
&\lesssim \|w\|^{\sigma,\ell}_{\delta,X} \sum_{ \lambda_1 \leq \ell}
\Bigl\| \bigl( 2^{ q  ( s-\lambda_1\delta)}
(T_{X\cdot\nabla})^{\lambda_1} \Delta_{q } v \bigr)_{\ell^2(\N)}
\Bigr\|_{L^p}
\\
&\lesssim
 |v |^{s,p,\ell}_{\delta,X}
\|w\|^{\sigma,\ell}_{\delta,X}.
\end{align*}
Notice  that
 when for $s=0,$  one has
$$
\Bigl\| \Bigl( \sum_{q_2\leq q+N} 2^{-(q-q_2)\lambda_1\delta} 2^{
-q_2 \lambda_1\delta} (T_{X\cdot\nabla})^{\lambda_1} \Delta_{q_2} v
\Bigr)_{\ell^2(\N)} \Bigr\|_{L^p} \lesssim \left\{\begin{array}{l}
\|v\|^{0,\ell}_{\delta,X}\ \ \quad\  \ \mbox{if}\ \
\lambda_1>0,\\
C_\epsilon \|v\|_{W^{\epsilon,p}}\ \ \mbox{if}\ \ \lambda_1=0.
\end{array}\right. $$ This yields the second estimate of
(\ref{IV3}).

To handle the estimate of (\ref{IV4}), we write
$$
\Delta_q (T_{X\cdot\nabla})^\ell R(v,w) =\Delta_q \sum_{q_1\geq q-N}
(T_{X\cdot\nabla})^\ell \Delta_{q_1}v\tilde\Delta_{q_1}w,
$$
  from which and (\ref{II1}), we infer
 \begin{align*}
 \bigl| \Delta_q (T_{X\cdot\nabla})^\ell R(v,w) \bigr|
 &\lesssim
 \sum_{q_1\geq q-N}
  M\bigl( (T_{X\cdot\nabla})^\ell (\Delta_{q_1}v\tilde\Delta_{q_1}w) \bigr)
  \\
  &\lesssim
 \sum_{\substack{q_1\geq q-N\\ n,\lambda_1+\lambda_2\leq \ell}}
  2^{q_1\delta(\ell-\lambda_1-\lambda_2)}
  \|(T_{X\cdot\nabla})^{\lambda_1} \Delta_{q_1}v \|_{L^\infty}
  M^{\gamma_n(\la_2)+1}\bigl( (T_{X\cdot\nabla})^{\lambda_2}\tilde\Delta_{q_1}w \bigr)
  \\
  &\lesssim
 \sum_{\substack{q_1\geq q-N\\ n,\lambda_2\leq \ell}}
  2^{-q_1(\sigma_1-(\ell-\lambda_2)\delta )} \|v\|^{\sigma_1,\ell}_{\delta,X}
  M^{\gamma_n(\la_2)+1}\bigl( (T_{X\cdot\nabla})^{\lambda_2}\tilde\Delta_{q_1}w \bigr).
 \end{align*}
Then we get, by applying \eqref{II2}, that
\begin{align*}
&\|(T_{X\cdot\nabla})^\ell
R(v,w)\|_{W^{\sigma_1+\sigma_2-\ell\delta,p}}
\\
&=\Bigl\| \bigl( 2^{q(\sigma_1+\sigma_2-\ell\delta)} \Delta_q
(T_{X\cdot\nabla})^\ell R(v,w) \bigr)_{\ell^2(\N)} \Bigr\|_{L^p}
\\
&\lesssim \|v\|^{\sigma_1,\ell}_{\delta,X} \sum_{n, \lambda_2\leq
\ell} \Bigl\| \Bigl( \sum_{q_1\geq q-N}
2^{(q-q_1)(\sigma_1+\sigma_2-\ell\delta)} 2^{q_1( \sigma_2
-\lambda_2 \delta )} M^{\gamma_n(\la_2)+1}\bigl(
(T_{X\cdot\nabla})^{\lambda_2}\tilde\Delta_{q_1}w
\Bigr)_{\ell^2(\N)} \Bigr\|_{L^p}
\\
&\lesssim \|v\|^{\sigma_1,\ell}_{\delta,X} \sum_{n, \lambda_2\leq
\ell} \Bigl\| \bigl(2^{q( \sigma_2 -\lambda_2 \delta )}
M^{\gamma_n(\la_2)+1}\bigl(
(T_{X\cdot\nabla})^{\lambda_2}\tilde\Delta_{q_1}w
\bigr)_{\ell^2(\N)} \Bigr\|_{L^p}\quad \hbox{if
}\sigma_1+\sigma_2>\ell\delta
\\
&\lesssim \|v\|^{\sigma_1,l}_{\delta,X}
|w|^{\sigma_2,p,\ell}_{\delta,X}.
\end{align*} This proves \eqref{IV4}.

\no$\bullet$\underline{Proof of (\ref{V2})}
 It is easy to observe that (\ref{V2}) holds for $\ell=0.$  Inductively we assume that \eqref{V2} holds for $\ell$ provided   that
 $
 \sigma\in ]\ell\delta,1[$
 or $
 \sigma\in ]\ell\delta-1,1[$
  if $\div X=0.
 $
 We next show that
 \beq\label{Apeq1}
 \|(T_{X\cdot\nabla})^{\ell+1} w\|_{W^{\sigma-(\ell+1)\delta,p}}
 \lesssim | w\|_{\delta,X}^{\s,p,\ell+1} \andf
 \|\d_X^{\ell+1} w\|_{W^{\s-(\ell+1)\delta,p}}
 \lesssim |w|^{\sigma,p,\ell+1}_{\delta,X},
 \eeq
provided that
\begin{equation}\label{cond:sigma}
\sigma \in ](\ell+1)\delta, 1[ \hbox{ or }\sigma\in
](\ell+1)\delta-1,1[ \hbox{ if}\ \ \div X=0.
\end{equation}
It is easy to observe that
\begin{align*}
(T_{X\cdot\nabla})^{\ell+1}w
&=(T_{X\cdot\nabla})^{\ell}(T_{X\cdot\nabla}-\d_X)w
+(T_{X\cdot\nabla})^{\ell}\d_X w.
\end{align*}
And by using Bony's decomposition, one has \beno \begin{split}
(\d_X-T_{X\cdot\nabla})w =& T_{\nabla w}\cdot X+R(X,\nabla w)\\
= & T_{\nabla w}\cdot X+\div R(X, w)-R(\div X,w).
\end{split}
\eeno Due to $\s<1,$ applying \eqref{IV3} yields \beno
\begin{split}
\|(T_{X\cdot\nabla})^{\ell}T_{\nabla w}\cdot
X\|_{W^{\sigma-(\ell+1)\delta,p}} \lesssim & |T_{\nabla w}\cdot
X|^{\sigma-\delta,p,\ell}_{\delta,X}\\
 \lesssim& |\na
w|^{\sigma-1,p,\ell}_{\delta,X}
\|X\|^{1-\delta,\ell}_{\delta,X}\lesssim  |
w|^{\sigma,p,\ell}_{\delta,X} \|X\|^{1-\delta,\ell}_{\delta,X}.
\end{split}
\eeno Similarly whenever, $\s>-1-(\ell+1)\delta,$ applying
\eqref{III2} and \eqref{IV4} leads to \beno
\begin{split}
\|(T_{X\cdot\nabla})^{\ell}\div R(X,
w)\|_{W^{\sigma-(\ell+1)\delta,p}} \lesssim & |\div R(X, w)|^{\sigma-\delta,p,\ell}_{\delta,X}\\
 \lesssim& | R(X, w)|^{1+\sigma-\delta,p,\ell}_{\delta,X}\lesssim  |
w|^{\sigma,p,\ell}_{\delta,X} \|X\|^{1-\delta,\ell}_{\delta,X}.
\end{split}
\eeno The same estimate holds for $(T_{X\cdot\nabla})^{\ell} R(\div
X, w)$ provided that $\s>(\ell+1)\delta.$ So that under the
assumption \eqref{cond:sigma},
 we have
\beq\label{Apeq5}
\|(T_{X\cdot\nabla})^{\ell}(T_{X\cdot\nabla}-\d_X)w\|_{W^{\sigma-(\ell+1)\delta,p}}
\lesssim  | w|^{\sigma,p,\ell}_{\delta,X}
\|X\|^{1-\delta,\ell}_{\delta,X}. \eeq While it follows from
inductive assumption, $| w|^{\sigma,p,\ell}_{\delta,X} \lesssim |
w\|^{\sigma,p,\ell}_{\delta, X},$  that
\begin{align*}
 &\|(T_{X\cdot\nabla})^{\ell}\d_X w\|_{W^{\sigma-(\ell+1)\delta,p}}
\lesssim |\d_X w|^{\sigma-\delta,p,\ell}_{\delta,X} \lesssim |\d_X
w\|^{\sigma-\delta,p,\ell}_{\delta,X} \lesssim |
w\|^{\sigma,p,\ell+1}_{\delta,X}.
\end{align*}
This together with \eqref{Apeq5} proves the first estimate of
\eqref{Apeq1}.

By the same manner, using the  inductive assumptions and
\eqref{Apeq5}, we deduce that  under the condition
\eqref{cond:sigma},
\begin{align*}
\|\d_X^{\ell+1} w\|_{W^{\sigma-(\ell+1)\delta,p}} &\lesssim
\|\d_X^{\ell}(\d_X
w-T_{X\cdot\nabla}w)\|_{W^{\sigma-(\ell+1)\delta,p}}
+\|\d_X^{\ell}T_{X\cdot\nabla}w\|_{W^{\sigma-(\ell+1)\delta,p}}
\\
&\lesssim |\d_X
w-T_{X\cdot\nabla}w|^{\sigma-\delta,p,\ell}_{\delta,X}
+|T_{X\cdot\nabla}w|^{\sigma-\delta,p,\ell}_{\delta,X}
\\
&\lesssim | w|^{\sigma,p,\ell}_{\delta,X}
\|X\|^{1-\delta,\ell}_{\delta,X} +|w|^{\sigma,p,\ell+1}_{\delta,X}
\lesssim |w|^{\sigma,p,\ell+1}_{\delta,X}.
\end{align*} This proves the second estimate of \eqref{Apeq1}.
Thus (\ref{V2}) follows. This completes the proof of Lemma
\ref{lem:prop}.
\end{proof}

\medskip

\noindent {\bf Acknowledgments.} We would like to thank Professor
Jean-Yves Chemin for careful reading and profitable suggestions on
the preliminary version of this paper. Part of this work was done
when we were visiting Morningside Center of the Academy of
Mathematics and Systems Sciences, CAS. We appreciate the hospitality
and the financial support from MCM. P. Zhang is partially supported
by NSF of China under Grant   11371347, the fellowship from Chinese
Academy of Sciences and innovation grant from National Center for
Mathematics and Interdisciplinary Sciences.

\end{document}